\documentclass[12pt,leqno]{amsart}
\usepackage{hyperref}
\usepackage{amsmath,amsthm,amscd,amssymb}
\usepackage{amsfonts}
\usepackage{enumerate}

\newcommand{\arxiv}[1]{\href{http://arxiv.org/#1}{arXiv:#1}}
\newcommand*{\mailto}[1]{\href{mailto:#1}{\nolinkurl{#1}}}

\newtheorem{theorem}{Theorem}[section]
\newtheorem{lemma}[theorem]{Lemma}

\numberwithin{equation}{section}

\newtheorem{re}{Remark}[section]
\newtheorem{prop}{Proposition}[section]
\newtheorem{theo}{Theorem}[section]
\newtheorem{lem}{Lemma}[section]
\newtheorem{col}{Corollary}[section]

\newcommand{\be}{\begin{equation}}
\newcommand{\ee}{\end{equation}}
\newcommand\bes{\begin{eqnarray}} \newcommand\ees{\end{eqnarray}}
\newcommand{\bess}{\begin{eqnarray*}}
\newcommand{\eess}{\end{eqnarray*}}
\newcommand{\D}{\displaystyle}

\newcommand{\bu}{{\bf u}}
\newcommand{\bH}{{\bf H}}
\newcommand{\bB}{{\bf B}}
\newcommand{\bw}{{\bf w}}
\newcommand{\equ}{\overset{\rm def}{=}}

\def\XXint#1#2#3{{\setbox0=\hbox{$#1{#2#3}{\int}$}
     \vcenter{\hbox{$#2#3$}}\kern-.5\wd0}}

\numberwithin{equation}{section}

\begin{document}

\title[Optimal decay for compressible MHD equations]{Optimal decay for the compressible MHD equations in the critical regularity framework}

\author[Q. Bie]{Qunyi Bie}
\address[Q. Bie]{College of Science $\&$  Three Gorges Mathematical Research Center, China Three Gorges University, Yichang 443002, PR China}
\email{\mailto{qybie@126.com}}

\author[Q. Wang]{Qiru Wang}
\address[Q. Wang]{School of Mathematics, Sun Yat-Sen University, Guangzhou 510275, PR China}
\email{\mailto{mcswqr@mail.sysu.edu.cn}}
\author[Z.-A. Yao]{Zheng-an Yao}
\address[Z.-A. Yao]{School of Mathematics, Sun Yat-Sen University, Guangzhou 510275, PR China}
\email{\mailto{mcsyao@mail.sysu.edu.cn}}

\date\today
\subjclass[2010]{}

\begin{abstract} In this paper, we study the large time behavior of solutions to the  compressible magnetohydrodynamic equations in the $L^p$-type critical Besov spaces. Precisely, we show that if the initial data in the low frequencies additionally belong to some Besov space $\dot{B}_{2,\infty}^{-\sigma_1}$ with $\sigma_1\in (1-N/2, 2N/p-N/2]$, then the $\dot{B}_{p,1}^0$ norm of the critical global solutions presents the optimal decay  $t^{-\frac{N}{2}(\frac{1}{2}-\frac{1}{p})-\frac{\sigma_1}{2}}$ for $t\rightarrow+\infty$. The pure energy argument without the spectral analysis is performed, which allows us
to remove the usual smallness assumption of low frequencies.

\end{abstract}

\maketitle

\section{Introduction} Magnetohydrodynamics (MHD) is concerned with the motion of conducting fluids in an electromagnetic field and has a very wide range of applications. In view of the dynamic motion of field and the magnetic field interacting strongly on each other, both the hydrodynamic and electrodynamic effects must be considered.  The compressible viscous MHD equations in the isentropic case take
the form (see, e.g.,\cite{cabannes1970theoretical,kulikovskiy1965,laudau1984})
\begin{equation}\label{1.1}
 \left\{\begin{array}{ll}\displaystyle\partial_t \rho+{\rm div}(\rho{\bf u})
=0,\\[1ex]
 \displaystyle\partial_t(\rho{\bf u})+{\rm div}(\rho{\bf u}\otimes{\bf u})+\nabla P(\rho)\\[1ex]
 \quad={\bf B}\cdot\nabla {\bf B}-\frac{1}{2}\nabla(|{\bf B}|^2)+{\rm div}(2\mu D({\bf u})+\lambda{\rm div}\bu\, {\rm Id}),\\[1ex]
 \displaystyle\partial_t{\bB}+({\rm div}\bu)\bB+\bu\cdot\nabla\bB
 -\bB\cdot\nabla\bu=\theta\Delta\bB,\,\,\,\,{\rm div}\bB=0,
 \end{array}
 \right.
 \end{equation}
for $(t,x)\in \mathbb{R}_+\times\mathbb{R}^N\,(N\geq 2)$. Here $\rho=\rho(t,x)\in \mathbb{R}_+$ is the density function of the fluid, ${\bf u}={\bf u}(t,x)\in \mathbb{R}^N$ is the velocity, and ${\bf B}={\bf B}(t,x)\in\mathbb{R}^{N}$ represents the magnetic field. The scalar function $P(\rho)\in\mathbb{R}$ is the pressure, which is an increasing and convex function in $\rho$. The notation $D(\bu)\equ \frac{1}{2}(\nabla\bu+\nabla\bu^T)$ stands for the deformation tensor. The density-dependent functions $\lambda$ and $\mu$ (the bulk and shear viscosities) are supposed to be smooth enough and to satisfy $\mu>0$ and $\lambda+2\mu>0$. The constant $\theta>0$ stands for the magnetic diffusivity acting as a magnetic diffusion coefficient of the magnetic field. The symbol $\otimes$ denotes the Kronecker tensor product such that ${\bf u}\otimes{\bf u}=({\bf u}_i{\bf u}_j)_{1\leq i,j \leq N}$. System \eqref{1.1} is supplemented
with the initial data
\begin{equation}\label{1.2}
  (\rho, \bu, \bB)|_{t=0}=(\rho_0(x), \bu_0(x), \bB_0(x)), \,\,x\in \mathbb{R}^N,
\end{equation}
and we focus on solutions that are close to some constant state $(\rho^\ast, {\bf 0}, \bB^\ast)$ with $\rho^\ast>0$ and the nonzero vector  $\bB^\ast\in \mathbb{R}^N$, at infinity.


There have been a lot of works on MHD by many physicists and mathematicians
due to its physical importance and mathematical challenges, see for example \cite{chen2002global, chen2003existence, ducomet2006the, fan2007vanishing, freistuhler1995existence, hoff2005uniqueness, kawashima1982smooth,wang2003large} and the references therein.   By exploiting an energy method in Fourier spaces, Umeda, Kawashima and Shizuta \cite{umeda1984on} first investigated a rather general class of symmetric hyperbolic-parabolic systems, and found that the dissipative mechanism inducing the optimal decay rates are just the same as that of heat kernel. As a direct application, they obtained such decay rate of solutions to system \eqref{1.1}-\eqref{1.2} (near the equilibrium state $(\rho^\ast, {\bf 0}, \bB^\ast)$). Subsequently, Kawashima \cite{kawashima1984systems} in his doctoral dissertation proved
the global existence of smooth solutions to \eqref{1.1}-\eqref{1.2} in the condition that  the initial data are small
in $H^3(\mathbb{R}^3)$. In addition, the author also derived the following fundamental $L^q$-$L^2$ decay estimate in
$H^3(\mathbb{R}^3)\cap L^q(\mathbb{R}^3)\, (1\leq q<2)$:
\be\label{1.10}
\|(\rho-\rho^\ast, \bu, {\bf B}-{\bf B}^\ast)\|_{L^2{\mathbb{R}^3}}\leq C(1+t)^{-\frac{3}{2}(\frac{1}{q}-\frac{1}{2})}.
\ee
Later on, still for data with high Sobolev regularity, there are a number of works on the long-time behavior of solution to the compressible MHD equations, see for example \cite{chentan2010global,gao2015long, li2011optimal, tan2013optimal, zhang2010some} and the references therein.


As regards global-in-time results, \emph{scaling invariance} plays a fundamental role.  Here we observe  that system \eqref{1.1} is invariant by the transformation
\be\nonumber
\tilde{\rho}(t,x)=\rho(l^2t, lx),~~\tilde{\bf u}(t, x)=l{\bf u}(l^2t, lx),~~\tilde{\bf B}(t, x)=l{\bf B}(l^2t, lx),
\ee
 up to a change of the pressure law $\tilde{P}=l^2P$. A critical space is a space in which the norm is invariant under the scaling
$
(\tilde{e},\tilde{\bf f}, \tilde{\bf g})(x)=(e(lx), l{\bf f}(lx), l{\bf g}(lx)).
$

 When ${\bf B}\equiv {\bf 0}$, system \eqref{1.1} becomes the  compressible Navier-Stokes equations. In the critical framework, there have been a lot of results for the compressible (or incompressible) Navier-Stokes equations, see for example \cite{cannone1997a,  charve2010global,  chen2010global, danchin2000global, danchin2015fourier, danchin2016incompressible, danchin2016optimal, fujita1964on,  haspot2011existence, kozono1994semilinear,  okita2014optimal,  xin2018optimal, xu2019a}. In particular, regarding the large time asymptotic behavior of strong solutions for the compressible Navier-Stokes equations, Okati \cite{okita2014optimal} performed low and high frequency decompositions and proved the time decay rate for strong solutions in the $L^2$  critical framework and in dimension $N\geq 3$. In the survey paper \cite{danchin2015fourier},
 Danchin proposed another description of the time decay which allows to proceed with dimension $N\geq 2$ in the $L^2$ critical framework. Recently,  Danchin and Xu \cite{danchin2016optimal} extended the method of \cite{danchin2015fourier} to get optimal time decay rate in the
 general $L^p$ type critical spaces and in any dimension $N\geq 2$. Later on,  Xu \cite{xu2019a} developed a general low-frequency condition for optimal decay estimates, where the regularity $\sigma_1$ of $\dot{B}_{2,\infty}^{-\sigma_1}$
 belongs to a whole range $(1-\frac{N}{2}, \frac{2N}{p}-\frac{N}{2}]$,  and the proof mainly depends on the refined time-weighted energy approach in the Fourier semi-group framework. Very recently, originated from the idea as in \cite{guo2012decay, strain2006almost}, Xin and Xu \cite{xin2018optimal} developed a new energy  argument to remove the usual smallness assumption of low frequencies studied in \cite{danchin2016optimal}.

As for system \eqref{1.1}-\eqref{1.2} with $\bB^\ast={\bf 0}$, Hao \cite{hao2011well} obtained the global well-posedness of strong solutions in $L^2$-type critical Besov spaces. Consequently,  the authors in \cite{bian2013well, bian2016local, jia2016well} studied the local existence and uniqueness of solutions in the critical $L^p$ framework. Very recently, Shi and Xu \cite{shi2019global} considered the perturbation around the constant equilibrium $(\rho^\ast, {\bf 0}, \bB^\ast)$ with $\bB^\ast\neq {\bf 0}$ and obtained the local and global well-posedness results in the critical $L^p$ framework, and here we list the global well-posedness of strong solutions to system \eqref{1.1} as follows.
\begin{theo}\label{th1.1} {\rm(}\cite{shi2019global}{\rm)}
Let $N\geq 2$ and $p$ fulfill
\begin{equation}\label{1.3}
  2\leq p\leq \min (4, 2N/(N-2))\,\,\, { and}, \,additionally, \,p\neq 4\,\,{if}\,\, N=2.
\end{equation}
Suppose that {\rm div}$\bB_0=0$, $P^\prime(\rho^\ast)>0$ and that \eqref{1.2} is satisfied.
There exists a small positive constant $c=c(p, \mu, \lambda, \theta, P, \rho^\ast, \bB^\ast)$
and a universal integer $j_0\in \mathbb{Z}$ such that if $a_0\equ \rho_0-\rho^\ast\in \dot{B}_{p,1}^{\frac{N}{p}}$, $\bH_0\equ \bB_0-\bB^\ast\in \dot{B}_{p,1}^{\frac{N}{p}-1}$
and if in addition $(a_0^\ell, \bu_0^\ell, \bH_0^\ell)\in \dot{B}_{2,1}^{\frac{N}{2}-1}$
(with the notation $z^\ell\equ \dot{S}_{k_0+1}z$ and $z^h=z-z^\ell$)
with
$$
\mathcal{X}_{p,0}\equ \|(a_0, \bu_0, \bH_0)\|_{\dot{B}_{2,1}^{\frac{N}{2}-1}}^\ell
+\|(\nabla a_0, \bu_0, \bH_0)\|_{\dot{B}_{p,1}^{\frac{N}{p}-1}}^h\leq c,
$$
then the Cauchy problem \eqref{1.1}-\eqref{1.2} admits a unique global-in-time solution $(\rho, \bu, \bB)$
with $\rho=\rho^\ast+a$, $\bB=\bB^\ast+\bH$ and $(a, \bu, \bH)$ in the space $X_p$ defined by
\begin{equation}\nonumber
\left.\begin{array}{ll}
(a,\bu,\bH)^\ell\in \widetilde{\mathcal{C}_b}(\mathbb{R}_+; \dot{B}_{2,1}^{\frac{N}{2}-1})\cap
L^1(\mathbb{R}_+; \dot{B}_{2,1}^{\frac{N}{2}+1}), \,\,\,a^h\in \widetilde{\mathcal{C}_b}(\mathbb{R}_+; \dot{B}_{p,1}^{\frac{N}{p}})\cap
L^1(\mathbb{R}_+; \dot{B}_{p,1}^{\frac{N}{p}}),\\[1ex]
(\bu, \bH)^h\in \widetilde{\mathcal{C}_b}(\mathbb{R}_+; \dot{B}_{p,1}^{\frac{N}{p}-1})\cap
L^1(\mathbb{R}_+; \dot{B}_{p,1}^{\frac{N}{p}+1}),
\end{array}
\right.
\end{equation}
where $s\in \mathbb{R}, 1\leq q\leq \infty$.

Furthermore, we get for some constant $C=C(p,\mu,\lambda, \theta, P, \rho^\ast, \bB^\ast)$,
$$
\mathcal{X}_p(t)\leq C\mathcal{X}_{p,0},
$$
for any $t>0$, where
\begin{equation}\label{1.4}
\begin{split}
\mathcal{X}_p(t)&\equ \|(a, \bu, \bH)\|_{\widetilde{L}^\infty(\dot{B}_{2,1}^{\frac{N}{2}-1})}^\ell+
\|(a, \bu, \bH)\|_{{L}^1(\dot{B}_{2,1}^{\frac{N}{2}+1})}^\ell
+\|a\|_{\widetilde{L}^\infty(\dot{B}_{p,1}^{\frac{N}{p}})}^h
+\|a\|_{{L}^1(\dot{B}_{p,1}^{\frac{N}{p}})}^h\\[1ex]
&\quad+\|(\bu,\bH)\|_{\widetilde{L}^\infty(\dot{B}_{p,1}^{\frac{N}{p}-1})}^h
+\|(\bu,\bH)\|_{{L}^1(\dot{B}_{p,1}^{\frac{N}{p}+1})}^h.
\end{split}
\end{equation}
\end{theo}

 The natural next problem is to explore the large time asymptotic behavior of global solutions constructed above. Shi and Xu \cite{shi2018large}
applied  Fourier analysis techniques to give precise description for the large time asymptotic behavior of solutions, not only in Lebesgue spaces but also in a full family of Besov spaces with negative regularity indexes. In this paper, motivated by the works \cite{guo2012decay, shi2018large, strain2006almost, xin2018optimal}, we intend to establish the optimal decay for the compressible MHD equations in the $L^p$ type critical framework without the smallness assumption of low frequencies.

\section{Main results}\label{s:2}
\setcounter{equation}{0}\setcounter{section}{2}\indent
 Let us first rewrite system \eqref{1.1} as the nonlinear perturbation form of constant equilibrium state $(\rho^\ast, {\bf 0}, \bB^\ast)$, looking at the nonlinearities as source terms.
To simplify the statement of main results, we assume that $\rho^\ast=1$, $\bB^\ast=I$ ($I$ is an arbitrary nonzero constant vector satisfying $|I|=1$), $P^\prime(\rho^\ast)=1$, $\theta=1$ and $\nu^\ast\equ 2\mu^\ast+\lambda^\ast=1$ (with $\mu^\ast\equ \mu(\rho^\ast)$ and $\lambda^\ast\equ \lambda(\rho^\ast)$). Consequently, in term of the new variables $(a, \bu, \bH)$, system \eqref{1.1} becomes
\begin{equation}\label{1.5}
\left\{\begin{array}{ll}
\partial_ta+{\rm div}\bu=f,\\[1ex]
\partial_t\bu-\mathcal{A}\bu+\nabla a+\nabla(I\cdot\bH)-I\cdot\nabla\bH={\bf g},\\[1ex]
\partial_t\bH-\Delta \bH+({\rm div}\bu)I-I\cdot\nabla\bu={\bf m},\\[1ex]
{\rm div}\bH=0,
\end{array}
\right.
\end{equation}
where
\begin{equation}\nonumber
\begin{split}
\left.
\begin{array}{ll}
&f\equ -{\rm div}(a\bu),\\[2ex]
&{\bf g}\equ -\bu\cdot\nabla\bu-\pi_1(a)\mathcal{A}\bu-\pi_2(a)\nabla a+\frac{1}{1+a}{\rm div}(2\widetilde\mu(a)
D(\bu)+\widetilde\lambda(a){\rm div}\bu\,{\rm Id})\\[1.5ex]\displaystyle
&\quad\quad+\pi_1(a)(\nabla(I\cdot\bH)-I\cdot\nabla\bH)-\frac{1}{1+a}(\frac{1}{2}\nabla
|\bH|^2-\bH\cdot\nabla\bH),\\[2ex]
&{\bf m}\equ -\bH({\rm div}\bu)+\bH\cdot\nabla\bu-\bu\cdot\nabla\bH,
\end{array}
\right.
\end{split}
\end{equation}
with
\begin{equation}\nonumber
\begin{array}{ll}\displaystyle
\mathcal{A}\equ \mu^\ast\Delta+(\lambda^\ast+\mu^\ast)\nabla{\rm div},\,\,{\rm here}\,\,2\mu^\ast
+\lambda^\ast=1\,\,{\rm and}\,\,\mu^\ast>0,\\[1ex]\displaystyle
\pi_1(a)\equ \frac{a}{1+a},\,\,\,\pi_2(a)\equ \frac{P^\prime(1+a)}{1+a}-1,\\[2ex]\displaystyle
\widetilde\mu(a)\equ \mu(1+a)-\mu(1),\,\,\,\widetilde\lambda(a)\equ \lambda(1+a)-\lambda(1).
\end{array}
\end{equation}
Note that $\pi_1, \pi_2, \widetilde\mu$ and $\widetilde\lambda$ are smooth functions satisfying
$$
\pi_1(0)=\pi_2(0)=\widetilde\mu(0)=\widetilde\lambda(0)=0.
$$

Denote $\Lambda^sf\equ \mathcal{F}^{-1}(|\xi|^s\mathcal{F}f)$ for $s\in\mathbb{R}$. Now, we state the
main results as follows.
\begin{theo}\label{th1.2}
Let $N\geq 2$ and $p$ satisfy assumption \eqref{1.3}. Let $(\rho,\bu, \bB)$ be the global solution addressed by Theorem \ref{th1.1}. If in addition
$(a_0, \bu_0, \bH_0)^\ell\in \dot{B}_{2,\infty}^{-\sigma_1}\,(1-\frac{N}{2}<\sigma_1\leq \sigma_0
\equ \frac{2N}{p}-\frac{N}{2})$ such that $\|(a,\bu_0,\bH_0)\|_{\dot{B}_{2,\infty}^{-\sigma_1}}^\ell$
is bounded, then we have
\begin{equation}\label{1.6}
\|(a,\bu,\bH)\|_{\dot{B}_{p,1}^\sigma}\lesssim (1+t)^{-\frac{N}{2}(\frac{1}{2}-\frac{1}{p})-\frac{\sigma+\sigma_1}{2}},
\end{equation}
where $-\sigma_1
-\frac{N}{2}+\frac{N}{p}<\sigma\leq \frac{N}{p}-1$ for all $t\geq 0$.
\end{theo}

By applying improved Gagliardo-Nirenberg inequalities, the  optimal  decay estimates of $\dot{B}_{2,\infty}^{-\sigma_1}$-$L^r$ type could be deduced as follows.
\begin{col}\label{col1}
Let those assumptions of Theorem \ref{th1.2} be fulfilled. Then the corresponding solution $(a,\bu,\bH)$ admits
\begin{equation}\label{1.7}
\|\Lambda^l(a,\bu,\bH)\|_{L^r}\lesssim (1+t)^{-\frac{N}{2}(\frac{1}{2}-\frac{1}{r})-\frac{l+\sigma_1}{2}},
\end{equation}
where $-\sigma_1
-\frac{N}{2}+\frac{N}{p}<l+\frac{N}{p}-\frac{N}{r}\leq \frac{N}{p}-1$ for $p\leq r\leq \infty$ and $t\geq 0$.
\end{col}

In the following, we give some comments.

\begin{re}
{\rm The low-frequency assumption of initial data in \cite{shi2018large} is at the endpoint $\sigma_0$ and the corresponding norm needs to be small enough, i.e., there exists a positive constant $c=c(p,\mu, \lambda, P, {\bf B}^\ast)$ such that
$
\|(a,\bu_0,\bH_0)\|_{\dot{B}_{2,\infty}^{-\sigma_0}}^\ell\leq c\,\,\,{\rm with}\,\,\,\sigma_0\equ \frac{2N}{p}-\frac{N}{2}.
$
Here, the new lower bound $1-\frac{N}{2}<\sigma_1\leq \sigma_0
$ enables us to enjoy larger freedom on the choice of $\sigma_1$, which allows to obtain more optimal decay estimates in the $L^p$ framework. In addition, the smallness of low frequencies is no longer needed in Theorem \ref{th1.2} and Corollary \ref{col1}.}
\end{re}
\begin{re}
 {\rm   In \cite{shi2018large}, there is a little loss on decay rates  due to the use of different Sobolev embeddings at low (or high) frequencies. For example, when $\sigma_1=\sigma_0$, the result in \cite{shi2018large} presents that the solution itself decays to equilibrium in $L^p$ norm with the rate of $O(t^{-\frac{N}{p}+\frac{N}{4}})$, which is no faster than that of  $O(t^{-\frac{N}{2p}})$ derived from  Corollary \ref{col1} above. }
 \end{re}
 \begin{re}
   {\rm  Condition \eqref{1.3} may allow  us to consider the case $p>N$, so that the regularity index $\frac{N}{p}-1$ of $(\bu,\bH)$ becomes negative in physical dimensions $N=2, 3$. Our result thus applies to large highly oscillating initial velocities
   and magnetic fields (see \cite{charve2010global, chen2010global} for more details). }
 \end{re}

 Let us give some illustration on the proof of main results.  Based on the works of \cite{danchin2000global, guo2012decay,  haspot2011existence, strain2006almost}, Xin and Xu \cite{xin2018optimal} developed a pure energy argument to establish the optimal decay for the barotropic compressible Navier-Stokes equations in the $L^p$ critical framework. Although the current proofs are in spirit of the works mentioned above,
 we have some new observations. More precisely, as pointed out in \cite{xin2018optimal}, the nonlinear estimates in the low frequencies (that is $\|(f,{\bf g}, {\bf m})\|_{\dot{B}_{2,\infty}^{-\sigma_1}}^\ell$) play an important role in the process of  proving Theorem \ref{th1.2}. They employed  different Sobolev embeddings and interpolations  to deal with the nonlinear terms in the non oscillation case $(2\leq p \leq N)$ and in the oscillation case $(p>N)$, respectively. Here,  we develop a new non-classical product estimate in the low frequencies (see \eqref{4.2} below), which enables us to unify  the estimates in the non oscillation case and the oscillation one. On the other hand, compared with \cite{xin2018optimal},  due to the appearance of the magnetic field, we need to take care of the  nonlinear estimates for those terms including the magnetic field. To the end, we make full use of the structure of the MHD equations itself.  For example, regarding the estimate of trinomial term $\frac{1}{1+a}(\frac{1}{2}\nabla |\bH|^2-\bH\cdot \nabla \bH)$, we  are going to take full advantage of its symmetrical structure (see \eqref{4.200}-\eqref{4.44}, \eqref{5.37} and \eqref{5.38} below).

 The rest of this paper is structured as follows. In  Section \ref{s:3}, we recall some basic properties of  the homogeneous Besov spaces. In Section \ref{s:4}, making use of  the pure energy arguments, we investigate the low-frequency and high-frequency estimates of solutions. Section \ref{s:5} is devoted to the estimation of $L^2$-type Besov norms at low frequencies, which plays the key role in deriving the Lyapunov-type inequality for energy norms. Section \ref{s:6}, i.e., the last section presents the proofs of Theorem \ref{th1.2} and Corollary \ref{col1}.

 Throughout the paper, $C$ stands for a harmless ``constant", and we sometimes write $A\lesssim B$ as an equivalent to $A\leq CB$. The notation $A\approx B$ means that $A\lesssim B$ and $B\lesssim A$. For any Banach space $X$ and $u, v\in X$, we agree that $\|(u,v)\|_X\equ \|u\|_X+\|v\|_X$. For  $p\in [1,+\infty]$ and $T>0$, the notation $L^p(0,T;X)$ or $L^p_T(X)$
designates the set of measurable functions $f:[0,T]\rightarrow X$ with $t\mapsto
\|f(t)\|_X$ in $L^p(0,T)$, endowed with the norm
$$
\|f\|_{L^p_T(X)}:=\bigl{\|}\|f\|_X\bigr{\|}_{L^p(0,T)}.
$$
We agree that $\mathcal{C}([0,T];X)$ denotes the set of continuous functions from
$[0,T]$ to $X$.
\section{Preliminaries}\label{s:3}
\setcounter{equation}{0}\setcounter{section}{3}\indent
We first recall the definition of homogeneous Besov spaces. They could be defined by using a dyadic partition of unity in Fourier variables called homogeneous Littlewood-Paley decomposition. To this end, choose a radial function $\varphi\in \mathcal{S}(\mathbb{R}^N)$ supported in $\mathcal{C}=\{\xi\in\mathbb{R}^N, \frac{3}{4}\leq |\xi|\leq \frac{8}{3}\}$ such that
$
\sum_{j\in\mathbb{Z}}\varphi(2^{-j}\xi)=1\quad\!\!{\rm if}\quad\!\!\xi\neq 0.
$
The homogeneous frequency localization operator $\dot{\Delta}_j$ and $\dot{S}_j$ are defined by
$$
\dot{\Delta}_j u=\varphi (2^{-j}D)u, \quad\,\dot{S}_j u=\sum_{k\leq j-1}\dot{\Delta}_k u\quad\,{\rm for}\quad\,j\in\mathbb{Z}.
$$
With our choice of $\varphi$, it is easy to see that
\be\label{1.8}
\dot{\Delta}_j\dot{\Delta}_kf=0\,\,\,{\rm if}\,\,\,|j-k|\geq 2,\,\,\,\,\,{\rm and}\,\,\,\,\dot{\Delta}_j(\dot{S}_{k-1}\dot{\Delta}_kf=0)
\,\,\,{\rm if}\,\,\,|j-k|\geq 5.
\ee

Let us denote the space $\mathcal{Y}^\prime(\mathbb{R}^N)$ by the quotient space of $\mathcal{S}^\prime(\mathbb{R}^N)/\mathcal{P}$ with the polynomials space $\mathcal{P}$. The formal equality
$
u=\sum_{k\in\mathbb{Z}}\dot{\Delta}_k u
$
holds true for $u\in \mathcal{Y}^\prime(\mathbb{R}^N)$ and is called the homogeneous Littlewood-Paley decomposition.

We then define, for $s\in\mathbb{R}$, $1\leq p, r\leq +\infty$, the homogeneous Besov space
$$
\dot{B}_{p,r}^s={\Big\{}f\in\mathcal{Y}^\prime(\mathbb{R}^N): \|f\|_{\dot{B}_{p,r}^s}<+\infty{\Big\}},
$$
where
$$
\|f\|_{\dot{B}_{p,r}^s}:=\|2^{ks}\|\dot{\Delta}_k f\|_{L^p}\|_{\ell^r}.
$$

When employing parabolic estimates in Besov spaces, it is somehow natural to take the time-Lebesgue norm before performing the summation for computing the Besov norm. So we next introduce the following Besov-Chemin-Lerner space $\widetilde{L}_T^\rho(\dot{B}_{p,r}^s)$ (see\,\cite{chemin1995flot}):
$$
\widetilde{L}_T^\rho(\dot{B}_{p,r}^s)={\Big\{}f\in (0,+\infty)\times\mathcal{Y}^\prime(\mathbb{R}^N):
\|f\|_{\widetilde{L}_T^\rho(\dot{B}_{p,r}^s)}<+\infty{\Big\}},
$$
where
$$
\|f\|_{\widetilde{L}_T^\rho(\dot{B}_{p,r}^s)}:=\bigl{\|}2^{ks}\|\dot{\Delta}_k f(t)\|_{L^\rho(0,T;L^p)}\bigr{\|}_{\ell^r}.
$$
The index $T$ will be omitted if $T=+\infty$ and we shall denote by $\widetilde{\mathcal{C}}_b([0,T]; \dot{B}^s_{p,r})$ the subset of functions of $\widetilde{L}^\infty_T(\dot{B}^s_{p,r})$ which are also continuous from
$[0,T]$ to $\dot{B}^s_{p,r}$.

A direct application of Minkowski's inequality implies that
$$
L_T^\rho(\dot{B}_{p,r}^s)\hookrightarrow \widetilde{L}_T^\rho(\dot{B}_{p,r}^s)\,\,\,{\rm if}\,\,\,r\geq \rho,
\quad\,{\rm and}\quad\,
\widetilde{L}_T^\rho(\dot{B}_{p,r}^s)\hookrightarrow {L}_T^\rho(\dot{B}_{p,r}^s)\,\,\,{\rm if}\,\,\,\rho\geq r.
$$
We will repeatedly use the following Bernstein's inequality throughout the paper:
\begin{lem}\label{le2.1}
{\rm(}see {\rm \cite{chemin1998perfect}}{\rm )} Let $\mathcal{C}$ be an annulus and $\mathcal{B}$ a ball, $1\leq p\leq q\leq +\infty$. Assume that $f\in L^p(\mathbb{R}^N)$, then for any nonnegative integer $k$, there exists constant $C$ independent of $f$, $k$ such that
$$
{\rm supp} \hat{f}\subset\lambda \mathcal{B}\Rightarrow\|D^k f\|_{L^q(\mathbb{R}^N)}:=\sup_{|\alpha|=k}\|\partial^\alpha f\|_{L^q(\mathbb{R}^N)}\leq C^{k+1}\lambda^{k+N(\frac{1}{p}-\frac{1}{q})}\|f\|_{L^p(\mathbb{R}^N)},
$$
\be\nonumber
{\rm supp} \hat{f}\subset\lambda\mathcal{C}\Rightarrow C^{-k-1}\lambda^k\|f\|_{L^p(\mathbb{R}^N)}\leq \|D^k f\|_{L^p(\mathbb{R}^N)}\leq
C^{k+1}\lambda^k\|f\|_{L^p(\mathbb{R}^N)}.
\ee
\end{lem}

More generally, if $v$ satisfies ${\rm Supp}\mathcal{F}v\subset\{\xi\in\mathbb{R}^N: R_1\lambda\leq |\xi|\leq R_2\lambda\}$
for some $0< R_1<R_2$ and $\lambda>0$, then for any smooth homogeneous of degree $m$ function $A$ on $\mathbb{R}^N\backslash\{0\}$ and $1\leq q\leq \infty$, it holds that (see e.g. Lemma 2.2 in \cite{bahouri2011fourier}):
\be\label{2.100}
\|A(D)v\|_{L^q}\lesssim \lambda^m\|v\|_{L^q}.
\ee

The following nonlinear generalization of \eqref{2.100} will be applied (see Lemma 8 in \cite{danchin2010well-posedness}):

\begin{prop}\label{pr2.3}
If ${\rm Supp}\mathcal{F}f\subset\{\xi\in\mathbb{R}^N: R_1\lambda\leq |\xi|\leq R_2\lambda\}$ then there exists $c$ depending only on $N, R_1$ and $R_2$ so that for all $1<p<\infty$,
$$
c\lambda^2\left(\frac{p-1}{p^2}\right)\int_{\mathbb{R}^N}|f|^pdx\leq (p-1)\int_{\mathbb{R}^N}|\nabla f|^2|f|^{p-2}dx
=-\int_{\mathbb{R}^N}\Delta f|f|^{p-2}fdx.
$$
\end{prop}

Let us now state some classical properties for the Besov spaces.
\begin{prop}\label{pr2.1} The following properties hold true:

\medskip
{\rm 1)} Derivation: There exists a universal constant $C$ such that
$$
C^{-1}\|f\|_{\dot{B}_{p,r}^s}\leq \|\nabla f\|_{\dot{B}_{p,r}^{s-1}}\leq C\|f\|_{\dot{B}_{p,r}^s}.
$$

{\rm 2)} Sobolev embedding: If $1\leq p_1\leq p_2\leq\infty$ and $1\leq r_1\leq r_2\leq\infty$, then $\dot{B}_{p_1, r_1}^s\hookrightarrow \dot{B}_{p_2, r_2}^{s-\frac{N}{p_1}+\frac{N}{p_2}}$.

{\rm 3)} Real interpolation: $\|f\|_{\dot{B}_{p,r}^{\theta s_1+(1-\theta)s_2}}\leq \|f\|_{\dot{B}_{p,r}^{s_1}}^{\theta}\|f\|_{\dot{B}_{p,r}^{s_2}}^{1-\theta}$.

{\rm 4)} Algebraic properties: for $s>0$, $\dot{B}_{p,1}^s\cap L^\infty$ is an algebra.

{\rm 5)} Scaling properties:

\quad\quad {\rm(a)} for all $\lambda>0$ and $f\in\dot{B}_{p,1}^s$, we have
$$
\|f(\lambda\cdot)\|_{\dot{B}_{p,1}^s}\approx \lambda^{s-\frac{N}{p}}\|f\|_{\dot{B}_{p,1}^s},
$$

\quad\quad {\rm(b)} for $f=f(t, x)$ in $L^r(0, T; \dot{B}_{p,1}^s)$, we have
$$
\|f(\lambda^a\cdot, \lambda^b\cdot)\|_{L_T^r(\dot{B}_{p,1}^s)}\approx
\lambda^{b(s-\frac{N}{p})-\frac{a}{r}}\|f\|_{L_{\lambda^aT}^r(\dot{B}_{p,1}^s)}.
$$
\end{prop}
Next we recall a few nonlinear estimates in Besov spaces which may be obtained
by means of paradifferential calculus. Firstly introduced by  Bony in \cite{bony1981calcul}, the paraproduct between $f$
and $g$ is defined by
$$
T_f g=\sum_{q\in\mathbb{Z}}\dot{S}_{q-1}f\dot{\Delta}_q g,
$$
and the remainder is given by
$$
R(f,g)=\sum_{q\in\mathbb{Z}}\dot{\Delta}_q f\widetilde{\dot{\Delta}}_q g\,\,\,\,{\rm with}\,\,\,\,\widetilde{\dot{\Delta}}_q g:=(\dot{\Delta}_{q-1}+\dot{\Delta}_{q}+\dot{\Delta}_{q+1})g.
$$
We have the following so-called Bony's decomposition:
\be\label{2.3}
fg=T_g f+T_f g+R(f,g).
\ee

The paraproduct $T$ and the remainder $R$ operators satisfy the following continuous properties (see e.g. \cite{bahouri2011fourier}).
 \begin{prop}\label{pr2.2}
Suppose that $s\in\mathbb{R}, \sigma>0,$ and $1\leq p, p_1, p_2, r, r_1, r_2\leq \infty$.  Then we have
 \medskip

{\rm 1)} The paraproduct $T$ is a bilinear, continuous operator from $L^\infty\times\dot{B}_{p,r}^s$ to $\dot{B}_{p,r}^s$, and from $\dot{B}_{\infty, r_1}^{-\sigma}\times\dot{B}_{p,r_2}^s$ to $\dot{B}_{p,r}^{s-\sigma}$ with
$\frac{1}{r}=\min\{1, \frac{1}{r_1}+\frac{1}{r_2}\}$.

\medskip
{\rm 2)} The remainder $R$ is bilinear continuous from $\dot{B}_{p_1,r_1}^{s_1}\times\dot{B}_{p_2,r_2}^{s_2}$ to $\dot{B}_{p,r}^{s_1+s_2}$ with $s_1+s_2>0$, $\frac{1}{p}=\frac{1}{p_1}+\frac{1}{p_2}\leq 1$, and $\frac{1}{r}=\frac{1}{r_1}+\frac{1}{r_2}\leq 1$.

 \end{prop}
 \medbreak
From \eqref{2.3} and Proposition \ref{pr2.2}, we may deduce the following two corollaries concerning the product estimates.
\begin{col}\label{co2.1} {\rm(}\cite{bahouri2011fourier}, \cite{danchin2002zero}{\rm )}
{\rm\,(i)} Let $s>0$ and $1\leq p,r\leq \infty$. Then $\dot{B}_{p,r}^s\cap L^\infty$ is an algebra and
$$
\|uv\|_{\dot{B}_{p,r}^s}\lesssim \|u\|_{L^\infty}\|v\|_{\dot{B}_{p,r}^s}+\|v\|_{L^\infty}\|u\|_{\dot{B}_{p,r}^s}.
$$

{\rm (ii)}\,If $u\in\dot{B}_{p_1,1}^{s_1}$ and $v\in\dot{B}_{p_2,1}^{s_2}$ with $1\leq p_1\leq p_2\leq \infty,~s_1\leq \frac{N}{p_1},~s_2\leq \frac{N}{p_2}$ and $s_1+s_2>0$, then
$uv\in\dot{B}_{p_2,1}^{s_1+s_2-\frac{N}{p_1}}$ and there exists a constant $C$, depending only on $N, s_1, s_2, p_1$ and $p_2$, such that
\be\label{2.1}
\|uv\|_{\dot{B}_{p_2,1}^{s_1+s_2-\frac{N}{p_1}}}\leq C\|u\|_{\dot{B}_{p_1,1}^{s_1}}
\|v\|_{\dot{B}_{p_2,1}^{s_2}}.
\ee
\end{col}

\begin{col}\label{co2.2}
Let $\sigma_1$ and $p$ satisfy the conditions as in Theorem \ref{th1.2}, that is, $1-\frac{N}{2}<\sigma_1\leq\frac{2N}{p}-\frac{N}{2}\,(N\geq 2)$ and  $p$ fulfills \eqref{1.3}, then we have
\be\nonumber
\|fg\|_{\dot{B}_{p,\infty}^{-\sigma_1+\frac{N}{p}-\frac{N}{2}+1}}\lesssim\|f\|_{\dot{B}_{p,1}^{\frac{N}{p}}}
\|g\|_{\dot{B}_{p,\infty}^{-\sigma_1+\frac{N}{p}-\frac{N}{2}+1}},
\ee
and
\be\nonumber
\|fg\|_{\dot{B}_{p,\infty}^{-\sigma_1+\frac{2N}{p}-N+1}}\lesssim\|f\|_{\dot{B}_{p,1}^{\frac{N}{p}}}
\|g\|_{\dot{B}_{p,\infty}^{-\sigma_1+\frac{2N}{p}-N+1}}.
\ee
\end{col}
Here, the estimates in Corollary \ref{co2.1} are classical,  and the  non-classical estimates in Corollary \ref{co2.2}  are used to establish the evolution of Besov norms at low frequencies in our paper.

We also need the following  composition lemma (see \cite{bahouri2011fourier,danchin2000global, runst1996sobolev}).
\begin{prop}\label{pra.4}

Let $F:\mathbb{R}\rightarrow \mathbb{R}$ be smooth with $F(0)=0$. For all $1\leq p,r\leq \infty$ and $s>0$, it holds that $F(u)\in \dot{B}_{p,r}^s\cap L^\infty$ for $u\in\dot{B}_{p,r}^s\cap L^\infty$, and
$$
\|F(u)\|_{\dot{B}_{p,r}^s}\leq C\|u\|_{\dot{B}_{p,r}^s}
$$
with $C$ depending only on $\|u\|_{L^\infty}$, $F^\prime$ (and higher derivatives), $s, p$ and $N$.

In the case $s>-\min(\frac{N}{p}, \frac{N}{p^\prime})$, then $u\in\dot{B}_{p,r}^s\cap\dot{B}_{p,1}^{\frac{N}{p}}$ implies that
$F(u)\in\dot{B}_{p,r}^s\cap\dot{B}_{p,1}^{\frac{N}{p}}$, and
$$
\|F(u)\|_{\dot{B}_{p,r}^s}\leq C(1+\|u\|_{\dot{B}_{p,1}^{\frac{N}{p}}})\|u\|_{\dot{B}_{p,r}^s},
$$
where $\frac{1}{p}+\frac{1}{p^\prime}=1$.
\end{prop}
The following commutator estimates (see \cite{danchin2016optimal}) have been employed in the high-frequency estimate for proving Theorem \ref{th1.2}.
\begin{prop}\label{pra.5}
Let $1\leq p, p_1\leq \infty$ and
\be\label{A4}
-\min\Big{\{}\frac{N}{p_1}, \frac{N}{p^\prime}\Big{\}}<\sigma\leq 1+
\min\Big{\{}\frac{N}{p}, \frac{N}{p_1}\Big{\}}.
\ee
There exists a constant $C>0$ depending only on $\sigma$ such that
for all $j\in\mathbb{Z}$ and $i\in \{1,\cdots, N\}$, we have
\be\label{A5}
\|[v\cdot\nabla,\partial_i{\dot{\Delta}_j}]a\|_{L^p}\leq Cc_j 2^{-j(\sigma-1)}
\|\nabla v\|_{\dot{B}_{p_1,1}^{\frac{N}{p_1}}}\|\nabla a\|
_{\dot{B}_{p,1}^{\sigma-1}},
\ee
where the commutator $[\cdot,\cdot]$ is defined by $[f,g]=fg-gf$, and
 $(c_j)_{j\in\mathbb{Z}}$ denotes a sequence such that $\|(c_j)\|_{\ell^1}\leq 1$ and $\frac{1}{p^\prime}+\frac{1}{p}=1$.
\end{prop}
Finally, we list the optimal regularity estimates for the heat equation (see e.g. \cite{bahouri2011fourier}).
\begin{prop}\label{pra.6}
Let $\sigma\in\mathbb{R},\,\, (p,r)\in [1,\infty]^2$ and $1\leq \rho_2\leq \rho_1\leq \infty$. Let $u$
satisfy
\be\label{A6}
\left\{\begin{array}{ll}\medskip\D
\partial_t u-\mu \Delta u=f,\\ \D
u|_{t=0}=u_0.
\end{array}
\right.
\ee
Then for all $T>0$, the following a prior estimate is satisfied:
\be\label{A7}
\mu^{\frac{1}{\rho_1}}\|u\|_{\widetilde{L}_T^{\rho_1}(\dot{B}_{p,r}^{\sigma+\frac{2}{\rho_1}})}
\lesssim \|u_0\|_{\dot{B}_{p,r}^\sigma}+\mu^{\frac{1}{\rho_2}-1}
\|f\|_{\widetilde{L}_T^{\rho_2}(\dot{B}_{p,r}^{\sigma-2+\frac{2}{\rho_2}})}.
\ee
\end{prop}

\section{Low-frequency and high-frequency estimates}\label{s:4}
In this section, we derive the low-frequency and high-frequency estimates to system \eqref{1.5}. Based on this, a Lyapunov-type inequality for energy norms could be deduced  in next section.
\subsection{Low-frequency estimates}
\begin{lemma}\label{le1}
Let $k_0$ be some integer. Then it holds that for all $t\geq 0$,
\begin{equation}\label{3.100}
\frac{d}{dt}\|(a,\bu,\bH)\|_{\dot{B}_{2,1}^{\frac{N}{2}-1}}^\ell
+\|(a,\bu,\bH)\|_{\dot{B}_{2,1}^{\frac{N}{2}+1}}^\ell\lesssim \|(f,{\bf g},{\bf m})\|_{\dot{B}_{2,1}^{\frac{N}{2}-1}}^\ell
\end{equation}
where
\begin{equation}\nonumber
\|z\|_{\dot{B}_{2,1}^s}^\ell\equ \sum_{k\leq k_0}2^{ks}\|\dot{\Delta}_kz\|_{L^2}\,\,\,{\rm for}\,\,\,s\in\mathbb{R}.
\end{equation}
\end{lemma}

\begin{proof} The proof of Lemma \ref{le1} is similar to that in \cite{shi2019global}.
Set
\be\label{3.104}
\omega=\Lambda^{-1}{\rm div}\bu,\,\, {\bf \Omega}=\Lambda^{-1}{\rm curl}\bu, \,\,\,{\rm and}\,\,\,
{\bf E}=\Lambda^{-1}{\rm curl}\bH,
\ee
where ${\rm curl{\bf v}}\equ (\partial_jv_i-\partial_iv_j)_{ij}$
is $N\times N$ matrix and Let $\Lambda^sz\equ \mathcal{F}^{-1}(|\xi|^s\mathcal{F}z)\,(s\in\mathbb{R})$. So system \eqref{1.5} becomes
\begin{equation}\label{3.1}
\left\{
\begin{array}{ll}
\partial_t a+\Lambda\omega=F,\\[1ex]
\partial_t\omega-\Delta\omega-\Lambda a-I\cdot{\rm div}{\bf E}=G,\\[1ex]
\partial_t{\bf \Omega}-\mu^\ast\Delta {\bf \Omega}-I\cdot\nabla {\bf E}={\bf L},\\[1ex]
\partial_t{\bf E}-\Delta{\bf E}+{\rm curl}(\omega I)-I\cdot\nabla{\bf \Omega}={\bf M},\\[1ex]
\bu=-\Lambda^{-1}\nabla\omega+\Lambda^{-1}{\rm div}{\bf\Omega},\,\,\bH=\Lambda^{-1}{\rm div}{\bf E},\,\,
{\rm div}\bH=0,
\end{array}
\right.
\end{equation}
where
\begin{equation}\label{3.101}
F=f, \,\,\,G=\Lambda^{-1}{\rm div}{\bf g},\,\,\,{\bf L}=\Lambda^{-1}{\rm curl}{\bf g},\,\,\,{\bf M}=\Lambda^{-1}{\rm curl}{\bf m}.
\end{equation}
Applying the operator $\dot{\Delta}_k$ to \eqref{3.1} and denoting $n_k\equ \dot{\Delta}_kn$, one has for all $k\in \mathbb{Z}$,
\begin{equation}\label{3.2}
\left\{
\begin{array}{ll}
\partial_t a_k+\Lambda\omega_k=F_k,\\[1ex]
\partial_t\omega_k-\Delta\omega_k-\Lambda a_k-I\cdot{\rm div}{\bf E}_k=G_k,\\[1ex]
\partial_t{\bf \Omega}_k-\mu^\ast\Delta {\bf \Omega}_k-I\cdot\nabla {\bf E}_k={\bf L}_k,\\[1ex]
\partial_t{\bf E}_k-\Delta{\bf E}_k+{\rm curl}(\omega_k I)-I\cdot\nabla{\bf \Omega_k}={\bf M}_k.
\end{array}
\right.
\end{equation}

Taking the $L^2$ scalar product of \eqref{3.2}$_1$ with $a_k$, \eqref{3.2}$_2$ with $\omega_k$,
 \eqref{3.2}$_3$ with ${\bf\Omega}_k$,  and \eqref{3.2}$_4$ with ${\bf E}_k$, we derive that
\begin{equation}\label{3.3}
\frac{1}{2}\frac{d}{dt}\|a_k\|_{L^2}^2+(\Lambda\omega_k,a_k)=(F_k, a_k),
\end{equation}
\begin{equation}\label{3.4}
\frac{1}{2}\frac{d}{dt}\|\omega_k\|_{L^2}^2+\|\Lambda \omega_k\|_{L^2}^2-(\Lambda a_k,\omega_k)
-(I\cdot{\rm div}{\bf E}_k, \omega_k)=(G_k, \omega_k),
\end{equation}
\begin{equation}\label{3.5}
\frac{1}{2}\frac{d}{dt}\|{\bf \Omega}_k\|_{L^2}^2+\mu^\ast\|\Lambda {\bf \Omega}_k\|_{L^2}^2
-(I\cdot\nabla{\bf E}_k, {\bf\Omega}_k)=({\bf L}_k, {\bf \Omega}_k),
\end{equation}
\begin{equation}\label{3.6}
\frac{1}{2}\frac{d}{dt}\|{\bf E}_k\|_{L^2}^2+\|\Lambda E_k\|_{L^2}^2+({\rm curl}(\omega_kI), {\bf E}_k)
-(I\cdot\nabla{\bf \Omega}_k, {\bf E}_k)=({\bf M}_k, {\bf E}_k).
\end{equation}
Noticing that
$$
(\Lambda \omega_k, a_k)=(\Lambda a_k, \omega_k),\,\,({\rm curl}(\omega_kI), {\bf E}_k)
=2(I\cdot{\rm div}{\bf E}_k, \omega_k)\,\, {\rm and}\,\,(I\cdot\nabla{\bf E}_k, {\bf\Omega}_k)
=-(I\cdot\nabla{\bf \Omega}_k, {\bf E}_k).
$$
Combing \eqref{3.3} to \eqref{3.6}, we have
\begin{equation}\label{3.7}
\begin{split}
&\quad\frac{1}{2}\frac{d}{dt}(\|a_k\|_{L^2}^2+\|\omega_k\|_{L^2}^2+\frac{1}{2}\|{\bf \Omega}_k\|_
{L^2}^2+\frac{1}{2}\|{\bf E}_k\|_{L^2}^2)+\|\Lambda\omega_k\|_{L^2}^2+\frac{1}{2}\mu^\ast\|\Lambda{\bf \Omega}_k\|_
{L^2}^2+\frac{1}{2}\|\Lambda{\bf E}_k\|_
{L^2}^2\\[1ex]
&=(F_k, a_k)+(G_k, \omega_k)+\frac{1}{2}({\bf L}_k, {\bf \Omega}_k)+\frac{1}{2}({\bf M}_k, {\bf E}_k).
\end{split}
\end{equation}

Taking the $L^2$ scalar product of \eqref{3.2}$_1$ with $\Lambda\omega_k$, \eqref{3.2}$_2$ with $\Lambda a_k$, and \eqref{3.2}$_1$ with $\Lambda^2 a_k$, we obtain, respectively, that
\begin{equation}\nonumber
\left.
\begin{array}{ll}
(\partial_t a_k,\Lambda\omega_k)+\|\Lambda\omega_k\|_{L^2}^2=(F_k,\Lambda\omega_k),\\[1ex]
(\partial_t\omega_k, \Lambda a_k)+(\Lambda_k^2\omega_k, \Lambda a_k)-\|\Lambda a_k\|_{L^2}^2-(I\cdot{\rm div}{\bf E}_k,\Lambda a_k)=(G_k,\Lambda a_k),\\[1ex]\displaystyle
\frac{1}{2}\frac{d}{dt}\|\Lambda a_k\|_{L^2}^2+(\Lambda \omega_k, \Lambda^2a_k)=(\Lambda F_k, \Lambda a_k)
\end{array}
\right.
\end{equation}
which yields
\begin{equation}\label{3.8}
\begin{split}
&\quad\frac{1}{2}\frac{d}{dt}\left(\|\Lambda a_k\|_{L^2}^2-2(a_k,\Lambda\omega_k)\right)+\|\Lambda a_k\|_
{L^2}^2-\|\Lambda\omega_k\|_{L^2}^2+(I\cdot{\rm div}{\bf E}_k, \Lambda a_k)\\[1ex]
&=(\Lambda F_k, \Lambda a_k)-(F_k, \Lambda \omega_k)-(G_k, \Lambda a_k).
\end{split}
\end{equation}
Set
$$
\mathcal{J}_k^2(t)\equ \|a_k\|_{L^2}^2+\|\omega_k\|_{L^2}^2+\frac{1}{2}\|{\bf \Omega}_k\|_{L^2}^2
+\frac{1}{2}\|{\bf E}_k\|_{L^2}^2+\gamma\left(\|\Lambda a_k\|_{L^2}^2-2(a_k, \Lambda \omega_k)\right)
$$
for some $\gamma>0$, we get from \eqref{3.7} and \eqref{3.8} that
\begin{equation}\label{3.105}
\begin{split}
&\quad\frac{1}{2}\frac{d}{dt}\mathcal{J}_k^2(t)+(1-\gamma)\|\Lambda\omega_k\|_{L^2}^2
+\frac{1}{2}\mu^\ast\|\Lambda{\bf \Omega}_k\|_{L^2}^2+\frac{1}{2}\|\Lambda{\bf E}_k\|_{L^2}^2
+\gamma\left(\|\Lambda a_k\|_{L^2}^2+(I\cdot{\rm div}{\bf E}_k, \Lambda a_k)\right)\\[1ex]
&=(F_k, a_k)+(G_k, \omega_k)+\frac{1}{2}({\bf L}_k, {\bf \Omega}_k)+\frac{1}{2}({\bf M}_k, {\bf E}_k)
+\gamma\Big[(\Lambda F_k, \Lambda a_k)-(F_k, \Lambda \omega_k)-(G_k, \Lambda a_k)\Big].
\end{split}
\end{equation}
It follows from Young's inequality that for $k\leq k_0$
\begin{equation}\label{3.103}
\mathcal{J}_k^2(t)\thickapprox\|(a_k, \Lambda a_k, \omega_k, {\bf \Omega}_k, {\bf E}_k)\|_{L^2}^2
\thickapprox\|(a_k, \omega_k, {\bf \Omega}_k, {\bf E}_k)\|_{L^2}^2.
\end{equation}
Consequently, in the low-frequency case,  we get from \eqref{3.105} that
\begin{equation}\label{3.9}
\frac{1}{2}\frac{d}{dt}\mathcal{J}_k^2+2^{2k}\mathcal{J}_k^2\lesssim\|(F_k, G_k, {\bf L}_k, {\bf M}_k)\|_{L^2}\mathcal{J}_k,
\end{equation}
which implies that
\begin{equation}\label{3.10}
\frac{d}{dt}\mathcal{J}_k+2^{2k}\mathcal{J}_k\lesssim\|(F_k, G_k, {\bf L}_k, {\bf M}_k)\|_{L^2}
\end{equation}
for $k\leq k_0$. Therefore, multiplying both sides by $2^{k(N/2-1)}$, summing up on $k\leq k_0$ and using \eqref{3.101} yield \eqref{3.100}.
\end{proof}
\subsection{High-frequency estimates} In the high-frequency regime, the term ${\rm div}(au)$ would cause a loss of one derivative as there is no smoothing effect for $a$. To get around  this difficulty, as in \cite{haspot2011existence}, we introduce the effective velocity
\begin{equation}\label{3.102}
\bw\equ \nabla (-\Delta)^{-1}(a-{\rm div}\bu).
\end{equation}
\begin{lemma}\label{le2}
Let $k_0$ be chosen suitably large. Then it holds that for all $t\geq 0$,
\begin{equation}\label{3.11}
\begin{split}
&\quad\frac{d}{dt}\|(\nabla a,\bu,\bH)\|_{\dot{B}_{p,1}^{\frac{N}{p}-1}}^h+\Big(\|\nabla a\|_{\dot{B}_{p,1}^{\frac{N}{p}-1}}^h
+\|(\bu,\bH)\|_{\dot{B}_{p,1}^{\frac{N}{p}+1}}^h\Big)\\[1ex]
&\lesssim \|f\|_{\dot{B}_{p,1}^{\frac{N}{p}-2}}^h+\|({\bf g},{\bf m})\|_{\dot{B}_{p,1}^{\frac{N}{p}-1}}^h
+\|\nabla\bu\|_{\dot{B}_{p,1}^{\frac{N}{p}}}\|a\|_{\dot{B}_{p,1}^{\frac{N}{p}}},
\end{split}
\end{equation}
where
\begin{equation}\nonumber
\|z\|_{\dot{B}_{2,1}^s}^h\equ \sum_{k\geq k_0+1}2^{ks}\|\dot{\Delta}_kz\|_{L^2}\,\,\,{\rm for}\,\,\,s\in\mathbb{R}.
\end{equation}
\end{lemma}
\begin{proof}
Let $\mathcal{P}\equ {\rm Id}+\nabla(-\Delta)^{-1}{\rm div}$ be the Leray projector onto divergence-free vector fields, and
$\bw$ be defined in \eqref{3.102}. Then from system \eqref{1.5}, we get that $\mathcal{P}\bu, \bH$ and $\bw$ satisfy a heat equation, and
$a$ satisfies a damped transport equation as follows.
\begin{equation}\label{3.12}
\left\{
\begin{array}{ll}
\partial_t\mathcal{P}\bu-\mu^\ast\Delta\mathcal{P}\bu=\mathcal{P}{\bf g}+I\cdot\nabla\bH,\\[1ex]
\partial_t\bH-\Delta\bH={\bf m}-({\rm div}\bw)I+I\cdot\nabla\bw-aI-I\cdot\nabla^2(-\Delta)^{-1}a+I\cdot\nabla\mathcal{P}\bu,\\[1ex]
\partial_t\bw-\Delta\bw=\nabla(-\Delta)^{-1}(f-{\rm div}{\bf g})+\bw-(-\Delta)^{-1}\nabla a-\nabla(I\cdot\bH),\\[1ex]
\partial_t a+a=-{\rm div}(a\bu)-{\rm div}\bw.
\end{array}
\right.
\end{equation}
Applying $\dot{\Delta}_k$ to \eqref{3.12}$_1$ yields for all $k\in\mathbb{Z}$,
$$
\partial_t\mathcal{P}\bu_k-\mu^\ast\Delta\mathcal{P}\bu_k=\mathcal{P}{\bf g}_k+I\cdot\nabla\bH_k.
$$
Then, multiplying each component of the above equation by $|(\mathcal{P}u_k)^i|^{p-2}(\mathcal{P}u_k)^i$ and integrating
over $\mathbb{R}^N$ gives for $i=1,2,\cdots,N$,
\be\nonumber
\begin{split}
&\quad\frac{1}{p}\frac{d}{dt}\|\mathcal{P}u_k^i\|_{L^p}^p-\mu^\ast
\int_{\mathbb{R}^N}\Delta(\mathcal{P}u_k)^i|(\mathcal{P}u_k)^i|^{p-2}(\mathcal{P}u_k)^idx\\[1ex]
&=\int_{\mathbb{R}^N}|(\mathcal{P}u_k)^i|^{p-2}(\mathcal{P}u_k)^i(\mathcal{P}g_k^i+I_j\partial_j H_k^i)dx.
\end{split}
\ee
Applying Proposition \ref{pr2.3} and summing on $i=1,2,\cdots,N$, we get for some constant $c_p$ depending only on $p$ that
\be\nonumber
\frac{1}{p}\frac{d}{dt}\|\mathcal{P}\bu_k\|_{L^p}^p+c_p\mu^\ast2^{2k}\|\mathcal{P}\bu_k\|_{L^p}^p\leq
(\|\mathcal{P}{\bf g}_k\|_{L^p}+C2^k\|\bH_k\|_{L^p})\|\mathcal{P}\bu_k\|_{L^p}^{p-1}
\ee
which leads to
\be\label{3.13}
\frac{d}{dt}\|\mathcal{P}\bu_k\|_{L^p}+c_p\mu^\ast2^{2k}\|\mathcal{P}\bu_k\|_{L^p}\leq
\|\mathcal{P}{\bf g}_k\|_{L^p}+C2^k\|\bH_k\|_{L^p}.
\ee
On the other hand, from \eqref{3.12}$_2$ and \eqref{3.12}$_3$, we argue exactly as for proving \eqref{3.13} and obtain that
\be\label{3.14}
\frac{d}{dt}\|\bH_k\|_{L^p}+c_p2^{2k}\|\bH_k\|_{L^p}\leq
\|{\bf m}_k\|_{L^p}+C2^k\|(\bw_k,\mathcal{P}\bu_k)\|_{L^p}+C2^{-k}\|\nabla a_k\|_{L^p}
\ee
and
\be\label{3.15}
\frac{d}{dt}\|\bw_k\|_{L^p}+c_p2^{2k}\|\bw_k\|_{L^p}\leq
C2^{-k}\|{f}_k\|_{L^p}+\|({\bf g}_k,\bw_k)\|_{L^p}+C2^k\|\bH_k\|_{L^p}+C2^{-2k}\|\nabla a_k\|_{L^p}.
\ee
Since the function $a$ fulfills the damped transport equation \eqref{3.12}$_4$, then performing the operator $\partial_i\dot{\Delta}_k$ to \eqref{3.12}$_4$ and denoting $R_k^i\equ [\bu\cdot\nabla, \partial_i\dot{\Delta}_k]a$, one has
\be\label{3.16}
\partial_t\partial_ia_k+\bu\cdot\nabla\partial_ia_k+\partial_ia_k=-\partial_i\dot{\Delta}_k(a{\rm div}\bu)
-\partial_i{\rm div}\bw_k+R_k^i,\,\,\,i=1,2,\cdots,N.
\ee
Multiplying both sides of \eqref{3.16} by $|\partial_ia_k|^{p-2}\partial_ia_k$, integrating on $\mathbb{R}^N$, and performing an integration by parts in the second term, we arrive at
\be\nonumber
\begin{split}
\frac{1}{p}\frac{d}{dt}\|\partial_ia_k\|_{L^p}^p+\|\partial_ia_k\|_{L^p}^p=\frac{1}{p}&\int_{\mathbb{R}^N}
{\rm div}\bu|\partial_ia_k|^pdx\\[1ex]
&\quad+\int_{\mathbb{R}^N}(R_k^i-\partial_i\dot{\Delta}_k(a{\rm div}\bu)
-\partial_i{\rm div}\bw_k)|\partial_ia_k|^{p-2}\partial_ia_kdx.
\end{split}
\ee
Summing up on $i=1,2,\cdots,N$ and applying H\"{o}lder and Bernstein inequalities imply
\be\label{3.17}
\begin{split}
\frac{1}{p}\frac{d}{dt}\|\nabla a_k\|_{L^p}^p+\|\nabla a_k\|_{L^p}^p\leq &\Big(\frac{1}{p}\|{\rm div}\bu\|_{L^\infty}
\|\nabla a_k\|_{L^p}+\|\nabla\dot{\Delta}_k(a{\rm div}\bu)\|_{L^p}\\[1ex]
&\quad\quad\quad\quad\quad\quad+C2^{2k}\|\bw_k\|_{L^p}+\|R_k\|_{L^p}
\Big)\|\nabla a_k\|_{L^p}^{p-1},
\end{split}
\ee
which leads to
\be\label{3.18}
\begin{split}
&\quad\frac{1}{p}\frac{d}{dt}\|\nabla a_k\|_{L^p}+\|\nabla a_k\|_{L^p}\\[1ex]
&\leq \frac{1}{p}\|{\rm div}\bu\|_{L^\infty}
\|\nabla a_k\|_{L^p}+\|\nabla\dot{\Delta}_k(a{\rm div}\bu)\|_{L^p}
+C2^{2k}\|\bw_k\|_{L^p}+\|R_k\|_{L^p}.
\end{split}
\ee
Adding \eqref{3.18} (multiplying by $\beta c_p$ for some $\beta>0$), \eqref{3.13}, \eqref{3.14}
and \eqref{3.15} together gives
\be\nonumber
\begin{split}
  &\quad\frac{d}{dt}\left(\|(\mathcal{P}\bu_k, \bw_k, \bH_k)\|_{L^p}+\beta c_p\|\nabla a_k\|_{L^p}\right)
  +c_p2^{2k}(\mu^\ast\|\mathcal{P}\bu_k\|_{L^p}+\|(\bw_k, \bH_k)\|_{L^p})+\beta c_p\|\nabla a_k\|_{L^p}\\[1ex]
  &\leq\|\mathcal{P}{\bf g}_k\|_{L^p}+C2^k\|\bH_k\|_{L^p}+\|{\bf m}_k\|_{L^p}+C2^k\|(\bw_k,\mathcal{P}\bu_k)\|_{L^p}+C2^{-k}\|\nabla a_k\|_{L^p}\\[1ex]
  &\quad+\beta c_p\left(\frac{1}{p}\|{\rm div}\bu\|_{L^\infty}
\|\nabla a_k\|_{L^p}+\|\nabla\dot{\Delta}_k(a{\rm div}\bu)\|_{L^p}
+C2^{2k}\|\bw_k\|_{L^p}+\|R_k\|_{L^p}\right)\\[1ex]
  &\quad+C2^{-k}\|{f}_k\|_{L^p}+\|({\bf g}_k,\bw_k)\|_{L^p}+C2^{-2k}\|\nabla a_k\|_{L^p}.
\end{split}
\ee
Choosing $k_0$ suitably large and $\beta$ sufficiently small, we deduce that there exists a constant $c_0>0$ such that
for all $k\geq k_0+1$,
\be\nonumber
\begin{split}
  &\quad\frac{d}{dt}\|(\mathcal{P}\bu_k, \bw_k, \bH_k,\nabla a_k)\|_{L^p}+c_0\left(2^{2k}\|(\mathcal{P}\bu_k,\bw_k,\bH_k)\|_{L^p}
  +\|\nabla a_k\|_{L^p}\right)\\[1ex]
  &\lesssim 2^{-k}\|f_k\|_{L^p}+\|({\bf m}_k, {\bf g}_k)\|_{L^p}+\|{\rm div}\bu\|_{L^\infty}
\|\nabla a_k\|_{L^p}+\|\nabla\dot{\Delta}_k(a{\rm div}\bu)\|_{L^p}
+\|R_k\|_{L^p}.
\end{split}
\ee
Since
$$
\bu=\bw-\nabla (-\Delta)^{-1}a+\mathcal{P}\bu,
$$
it follows that
\be\nonumber
\begin{split}
  &\quad\frac{d}{dt}\|(\nabla a_k, \bu_k, \bH_k)\|_{L^p}+c_0\|(\nabla a_k, 2^{2k}\bu_k, 2^{2k}\bH_k)\|_{L^p}\\[1ex]
  &\lesssim \|(2^{-k}f_k, {\bf m}_k, {\bf g}_k)\|_{L^p}+\|{\rm div}\bu\|_{L^\infty}
\|\nabla a_k\|_{L^p}+\|\nabla\dot{\Delta}_k(a{\rm div}\bu)\|_{L^p}
+\|R_k\|_{L^p}.
\end{split}
\ee
Thus, multiplying by $2^{k(\frac{N}{p}-1)}$, summing up over $k\geq k_0+1$ and applying Corollary \ref{co2.1}
and Proposition \ref{pra.5}, we conclude \eqref{3.11}.
\end{proof}
\section{Estimation of $L^2$-type Besov norms at low frequencies}\label{s:5}
\begin{prop}\label{pr4.1}
Let $1-\frac{N}{2}<\sigma_1\leq \frac{2N}{p}-\frac{N}{2} (N\geq 2)$ and $p$ satisfy \eqref{1.3}.
Then the following two estimates  hold true:
\be\label{4.1}
\|fg\|_{\dot{B}_{2,\infty}^{-\sigma_1}}\lesssim\|f\|_{\dot{B}_{p,1}^{\frac{N}{p}}}\|g\|_{\dot{B}_{2,\infty}^{-\sigma_1}},
\ee
and
\be\label{4.2}
\|fg\|_{\dot{B}_{2,\infty}^{-\sigma_1}}^\ell\lesssim\|f\|_{\dot{B}_{p,1}^{\frac{N}{p}-1}}
\left(\|g\|_{\dot{B}_{p,\infty}^{-\sigma_1+\frac{N}{p}-\frac{N}{2}+1}}
+\|g\|_{\dot{B}_{p,\infty}^{-\sigma_1+\frac{2N}{p}-N+1}}\right).
\ee
\end{prop}
\begin{proof} Denote $p^\ast\equ\frac{2p}{p-2}$, i.e., $\frac{1}{p}+\frac{1}{p^\ast}=\frac{1}{2}$.
By \eqref{2.3}, we decompose $fg$ into $T_fg+R(f,g)+T_gf$.

Firstly, we prove \eqref{4.1}. Thanks to \eqref{1.8}, we have
\begin{equation}\label{4.3}
\begin{split}
\|\dot{\Delta}_j(T_fg)\|_{L^2}&=\|\sum_{|k-j|\leq 4}\dot{\Delta}_j(\dot{S}_{k-1}f\dot{\Delta}_kg)\|_{L^2}
=\|\sum_{|k-j|\leq 4}\sum_{k^\prime\leq k-2}\dot{\Delta}_j(\dot{\Delta}_{k^\prime}f\dot{\Delta}_kg)\|_{L^2}\\[1ex]
&\lesssim\sum_{|k-j|\leq 4}\sum_{k^\prime\leq k-2}\|\dot{\Delta}_{k^\prime}f\|_{L^\infty}\|\dot{\Delta}_kg\|_{L^2}\\[1ex]
&\lesssim\sum_{|k-j|\leq 4}\sum_{k^\prime\leq k-2}2^{k^\prime\frac{N}{p}}\|\dot{\Delta}_{k^\prime}f\|_{L^p}2^{k\sigma_1}2^{-k\sigma_1}\|\dot{\Delta}_kg\|_{L^2}\\[1ex]
&\lesssim 2^{j\sigma_1}\|f\|_{\dot{B}_{p,1}^{\frac{N}{p}}}\|g\|_{\dot{B}_{2,\infty}^{-\sigma_1}}.
\end{split}
\end{equation}
For the remainder term, one gets
\begin{equation}\label{4.4}
\begin{split}
\|\dot{\Delta}_jR(f,g)\|_{L^2}&=\|\sum_{k\geq j-3}\sum_{|k-k^\prime|\leq 1}\dot{\Delta}_j(\dot{\Delta}_{k}f\dot{\Delta}_{k^\prime}g)\|_{L^2}\leq \sum_{k\geq j-3}\sum_{|k-k^\prime|\leq 1}
\|\dot{\Delta}_j(\dot{\Delta}_{k}f\dot{\Delta}_{k^\prime}g)\|_{L^2}\\[1ex]
&\lesssim 2^{j\frac{N}{p}}\sum_{k\geq j-3}\sum_{|k-k^\prime|\leq 1}\|\dot{\Delta}_{k}f\dot{\Delta}_{k^\prime}g\|_{L^{\frac{2p}{p+2}}}\\[1ex]
&\lesssim 2^{j\frac{N}{p}}\sum_{k\geq j-3}\sum_{|k-k^\prime|\leq 1}2^{-k\frac{N}{p}}2^{k\frac{N}{p}}\|\dot{\Delta}_{k}f\|_{L^{p}}2^{k^\prime\sigma_1}
2^{-k^\prime\sigma_1}\|\dot{\Delta}_{k^\prime}g\|_{L^{2}}\\[1ex]
&\lesssim2^{j\frac{N}{p}}\sum_{k\geq j-3}2^{k(\sigma_1-\frac{N}{p})}c(k)\|f\|_{\dot{B}_{p,1}^{\frac{N}{p}}}\|g\|_{\dot{B}_{2,\infty}^{-\sigma_1}}\lesssim2^{j\sigma_1}\|f\|_{\dot{B}_{p,1}^{\frac{N}{p}}}\|g\|_{\dot{B}_{2,\infty}^{-\sigma_1}},
\end{split}
\end{equation}
here $\|c(k)\|_{l^1}=1$ and we  used that $\sigma_1-\frac{N}{p}\leq 0$ as $\sigma_1\leq \frac{2N}{p}-\frac{N}{2}\leq \frac{N}{p}$ in the last inequality.

For the  term $T_gf$,  it follows that
\begin{equation}\label{4.5}
\begin{split}
\|\dot{\Delta}_j(T_gf)\|_{L^2}&=\|\sum_{|k-j|\leq 4}\dot{\Delta}_j(\dot{S}_{k-1}g\dot{\Delta}_kf)\|_{L^2}
=\|\sum_{|k-j|\leq 4}\sum_{k^\prime\leq k-2}\dot{\Delta}_j(\dot{\Delta}_{k^\prime}g\dot{\Delta}_kf)\|_{L^2}\\[1ex]
&\lesssim\sum_{|k-j|\leq 4}\sum_{k^\prime\leq k-2}\|\dot{\Delta}_{k^\prime}g\|_{L^{p^\ast}}\|\dot{\Delta}_kf\|_{L^p}\\[1ex]
&\lesssim\sum_{|k-j|\leq 4}\sum_{k^\prime\leq k-2}2^{k^\prime(\frac{N}{p}+\sigma_1)}2^{-k^\prime\sigma_1}\|\dot{\Delta}_{k^\prime}g\|_{L^2}2^{-k\frac{N}{p}}2^{k\frac{N}{p}}\|\dot{\Delta}_kf\|_{L^p}\\[1ex]
&\lesssim 2^{j\sigma_1}\|f\|_{\dot{B}_{p,1}^{\frac{N}{p}}}\|g\|_{\dot{B}_{2,\infty}^{-\sigma_1}},
\end{split}
\end{equation}
here $\sigma_1+\frac{N}{p}>0$ since  $\sigma_1>1-\frac{N}{2}\geq -\frac{N}{p}$ if $p\leq \frac{2N}{N-2}$. Combining \eqref{4.3}, \eqref{4.4} and \eqref{4.5}, we finish the proof of \eqref{4.1}.

Now, we are in a position to prove \eqref{4.2}. For the paraproduct term $T_fg$, we have
\begin{equation}\label{4.6}
\begin{split}
\|\dot{\Delta}_j(T_fg)\|_{L^2}
&\leq\sum_{|k-j|\leq 4}\sum_{k^\prime\leq k-2}\|\dot{\Delta}_j(\dot{\Delta}_{k^\prime}f\dot{\Delta}_kg)\|_{L^2}\lesssim\sum_{|k-j|\leq 4}\sum_{k^\prime\leq k-2}\|\dot{\Delta}_{k^\prime}f\|_{L^{p^\ast}}\|\dot{\Delta}_kg\|_{L^p}\\[1ex]
&\lesssim\sum_{|k-j|\leq 4}\sum_{k^\prime\leq k-2}2^{k^\prime(\frac{2N}{p}-\frac{N}{2})}
\|\dot{\Delta}_{k^\prime}f\|_{L^p}\|\dot{\Delta}_kg\|_{L^p}\\[1ex]
&\lesssim\sum_{|k-j|\leq 4}\sum_{k^\prime\leq k-2}2^{k^\prime(\frac{2N}{p}-\frac{N}{2}+1-\frac{N}{p})}2^{k^\prime(\frac{N}{p}-1)}
\|\dot{\Delta}_{k^\prime}f\|_{L^p}\\[1ex]
&\quad\quad\quad\quad\quad\quad\quad\quad\quad\quad\quad\quad\times
2^{k(\sigma_1-\frac{N}{p}+\frac{N}{2}-1)}2^{-k(\sigma_1-\frac{N}{p}+\frac{N}{2}-1)}\|\dot{\Delta}_kg\|_{L^p}\\[1ex]
&\lesssim 2^{j\sigma_1}\|f\|_{\dot{B}_{p,1}^{\frac{N}{p}-1}}\|g\|_{\dot{B}_{p,\infty}^{-\sigma_1+\frac{N}{p}-\frac{N}{2}+1}},
\end{split}
\end{equation}
where we have used that $1+\frac{N}{p}-\frac{N}{2}\geq 0$ and $p^\ast\geq p$ as $p$ fulfills $2\leq p\leq \min(4,\frac{2N}{N-2})$.

For the remainder term, one gets
\begin{equation}\label{4.7}
\begin{split}
\|\dot{\Delta}_jR(f,g)\|_{L^2}&\leq \sum_{k\geq j-3}\sum_{|k-k^\prime|\leq 1}
\|\dot{\Delta}_j(\dot{\Delta}_{k}f\dot{\Delta}_{k^\prime}g)\|_{L^2}\\[1ex]
&\lesssim 2^{j(\frac{2N}{p}-\frac{N}{2})}\sum_{k\geq j-3}\sum_{|k-k^\prime|\leq 1}2^{k(1-\frac{N}{p})}2^{k(\frac{N}{p}-1)}\|\dot{\Delta}_{k}f\|_{L^{p}}\\[1ex]
&\quad\quad\quad\quad\quad\quad\quad\quad\quad\quad\quad\times2^{k^\prime(\sigma_1-\frac{N}{p}+\frac{N}{2}-1)}
2^{-k^\prime(\sigma_1-\frac{N}{p}+\frac{N}{2}-1)}\|\dot{\Delta}_{k^\prime}g\|_{L^{p}}\\[1ex]
&\lesssim2^{j(\frac{2N}{p}-\frac{N}{2})}\sum_{k\geq j-3}2^{k(\sigma_1-\frac{2N}{p}+\frac{N}{2})}c(k)\|f\|_{\dot{B}_{p,1}^{\frac{N}{p}-1}}
\|g\|_{\dot{B}_{p,\infty}^{-\sigma_1+\frac{N}{p}-\frac{N}{2}+1}}\\[1ex]
&
\lesssim2^{j\sigma_1}\|f\|_{\dot{B}_{p,1}^{\frac{N}{p}-1}}\|g\|_{\dot{B}_{p,\infty}^{-\sigma_1+\frac{N}{p}-\frac{N}{2}+1}},
\end{split}
\end{equation}
here $\|c(k)\|_{l^1}=1$ and  we have used the condition $\sigma_1\leq \frac{2N}{p}-\frac{N}{2}$ in the last inequality.

For the  term $T_gf$,  we could obtain
\begin{equation}\label{4.8}
\begin{split}
\|\dot{\Delta}_j(T_gf)\|_{L^2}&\leq \sum_{|k-j|\leq 4}\sum_{k^\prime\leq k-2}\|\dot{\Delta}_j(\dot{\Delta}_{k^\prime}g\dot{\Delta}_kf)\|_{L^2}\leq\sum_{|k-j|\leq 4}\sum_{k^\prime\leq k-2}\|\dot{\Delta}_{k^\prime}g\|_{L^{p^\ast}}\|\dot{\Delta}_kf\|_{L^p}\\[1ex]
&\lesssim\sum_{|k-j|\leq 4}\sum_{k^\prime\leq k-2}2^{k^\prime(\frac{2N}{p}-\frac{N}{2})}\|\dot{\Delta}_{k^\prime}g\|_{L^p}\|\dot{\Delta}_kf\|_{L^p}\\[1ex]
&\lesssim\sum_{|k-j|\leq 4}\sum_{k^\prime\leq k-2}2^{k^\prime(\frac{2N}{p}-\frac{N}{2}+\sigma_1-\frac{2N}{p}+N-1)}2^{k^\prime(-\sigma_1+\frac{2N}{p}-N+1)}
\|\dot{\Delta}_{k^\prime}g\|_{L^p}\\[1ex]
&\quad\quad\quad\quad\quad\quad\quad\quad\quad\quad\quad\quad\quad\quad\quad\quad\quad\quad\quad
\times2^{k(1-\frac{N}{p})}2^{k(\frac{N}{p}-1)}
\|\dot{\Delta}_kf\|_{L^p}\\[1ex]
&\lesssim 2^{j(\sigma_1+\frac{N}{2}-\frac{N}{p})}
\|f\|_{\dot{B}_{p,1}^{\frac{N}{p}-1}}\|g\|_{\dot{B}_{p,\infty}^{-\sigma_1+\frac{2N}{p}-N+1}},
\end{split}
\end{equation}
where we used that $\sigma_1>1-\frac{N}{2}$ in the last inequality.

From \eqref{4.6} and \eqref{4.7}, we deduce that
\be\label{4.9}
\|T_fg+R(f,g)\|_{\dot{B}_{2,\infty}^{-\sigma_1}}\lesssim\|f\|_{\dot{B}_{p,1}^{\frac{N}{p}-1}}\|g\|_{\dot{B}_{p,\infty}^{-\sigma_1+\frac{N}{p}-\frac{N}{2}+1}}
\ee
and from \eqref{4.8}, we get for $p\geq 2$ that
\be\label{4.10}
\|T_gf\|_{\dot{B}_{2,\infty}^{-\sigma_1}}^\ell\leq
\|T_gf\|_{\dot{B}_{2,\infty}^{-\sigma_1+\frac{N}{p}-\frac{N}{2}}}^\ell
\lesssim\|f\|_{\dot{B}_{p,1}^{\frac{N}{p}-1}}\|g\|_{\dot{B}_{p,\infty}^{-\sigma_1+\frac{2N}{p}-N+1}}.
\ee
Combining \eqref{4.9} and \eqref{4.10}, we get \eqref{4.2}.
\end{proof}

Next, we begin to estimate the $L^2$-type Besov norms at low frequencies, which is the main ingredient in the proof of Theorem \ref{th1.2}.
\begin{lemma}\label{le3}
Let $1-\frac{N}{2}<\sigma_1\leq \frac{2N}{p}-\frac{N}{2}$ and $p$ satisfy \eqref{1.3}, it holds that
\begin{equation}\label{4.11}
\begin{split}
&\Big(\|(a, \bu,\bH)(t)\|_{\dot{B}_{2,\infty}^{-\sigma_1}}^\ell\Big)^2\lesssim
\Big(\|(a_0, \bu_0,\bH_0)\|_{\dot{B}_{2,\infty}^{-\sigma_1}}^\ell\Big)^2\\[1ex]
&\quad\quad\quad\quad+\int_0^tA_1(\tau)\Big(\|(a, \bu,\bH)(\tau)\|_{\dot{B}_{2,\infty}^{-\sigma_1}}^\ell\Big)^2d\tau
+\int_0^tA_2(\tau)\|(a, \bu,\bH)(\tau)\|_{\dot{B}_{2,\infty}^{-\sigma_1}}^\ell d\tau,
\end{split}
\end{equation}
where
\begin{equation}\nonumber
\begin{split}
A_1(t)&\equ\|(a,\bu,\bH)\|_{\dot{B}_{2,1}^{\frac{N}{2}+1}}^\ell+\|a\|_{\dot{B}_{p,1}^{\frac{N}{p}}}^h
+\|(\bu, \bH)\|_{\dot{B}_{p,1}^{\frac{N}{p}+1}}^h\\[1ex]
&\quad\quad\quad\quad\quad\quad\quad\quad\quad\quad\quad\quad+\|a\|_{\dot{B}_{p,1}^{\frac{N}{p}}}^2+\|a\|_{\dot{B}_{p,1}^{\frac{N}{p}}}\|(\bu, \bH)\|_{\dot{B}_{p,1}^{\frac{N}{p}+1}}+\|a\|_{\dot{B}_{p,1}^{\frac{N}{p}}}\|\bH\|_{\dot{B}_{p,1}^{\frac{N}{p}}}
\end{split}
\end{equation}
and
\begin{equation}\nonumber
\begin{split}
A_2(t)&\equ\Big(\|(a,\bu,\bH)\|_{\dot{B}_{p,1}^{\frac{N}{p}}}^h\Big)^2
+\|(\bu,\bH)\|_{\dot{B}_{p,1}^{\frac{N}{p}+1}}^h\|a\|_{\dot{B}_{p,1}^{\frac{N}{p}}}\|a\|_{\dot{B}_{p,1}^{\frac{N}{p}}}^h
+\|a\|_{\dot{B}_{p,1}^{\frac{N}{p}}}^2\|a\|_{\dot{B}_{p,1}^{\frac{N}{p}}}^h\\[1ex]
&\quad\quad\quad\quad\quad\quad\quad\quad\quad\quad\quad\quad\quad\quad\quad+\|a\|_{\dot{B}_{p,1}^{\frac{N}{p}}}^h\|(\bu, \bH)\|_{\dot{B}_{p,1}^{\frac{N}{p}+1}}^h
+\Big(\|\bH\|_{\dot{B}_{p,1}^{\frac{N}{p}}}^h\Big)^2\|a\|_{\dot{B}_{p,1}^{\frac{N}{p}}}.
\end{split}
\end{equation}
\end{lemma}
\begin{proof}
 From  \eqref{3.103} and \eqref{3.9}, we have for $k\leq k_0$,
 \begin{equation}\label{4.12}
 \begin{split}
 &\quad\frac{1}{2}\frac{d}{dt}\|(a_k, \omega_k, {\bf \Omega}_k, {\bf E}_k)\|_{L^2}^2+\|(\Lambda a_k,\Lambda \omega_k, \Lambda {\bf \Omega}_k, \Lambda {\bf E}_k)\|_{L^2}^2\\[1ex]
 &\lesssim \|(F_k, G_k, {\bf L}_k, {\bf M}_k)\|_{L^2}\|(a_k, \omega_k, {\bf \Omega}_k, {\bf E}_k)\|_{L^2}.
 \end{split}
 \end{equation}
Multiplying $2^{2k(-\sigma_1)}$ on both sides of \eqref{4.12}, taking supremum in terms of $k\leq k_0$, integrating
over $[0, t]$ and noticing that \eqref{3.104} and \eqref{3.101}, we arrive at
\begin{equation}\label{4.13}
\begin{split}
&\quad\Big(\|(a, \bu,\bH)(t)\|_{\dot{B}_{2,\infty}^{-\sigma_1}}^\ell\Big)^2\\[1ex]
&\lesssim
\Big(\|(a_0, \bu_0,\bH_0)\|_{\dot{B}_{2,\infty}^{-\sigma_1}}^\ell\Big)^2+\int_0^t\|(f, {\bf g},{\bf m})(\tau)\|_{\dot{B}_{2,\infty}^{-\sigma_1}}^\ell\|(a, \bu,\bH)(\tau)\|_{\dot{B}_{2,\infty}^{-\sigma_1}}^\ell  d\tau.
\end{split}
\end{equation}

Next, we focus on the estimates of $\|(f, {\bf g},{\bf m})\|_{\dot{B}_{2,\infty}^{-\sigma_1}}^\ell$. Firstly, we deal with the term $f=-{\rm div}(a\bu)=-a{\rm div}\bu-\bu\cdot\nabla a$.

\underline{Estimate of $a{\rm div}\bu$}. We decompose
$$a{\rm div}\bu=a^\ell{\rm div}\bu+a^h{\rm div}\bu^\ell+a^h{\rm div}\bu^h.$$
Making use of  \eqref{4.1}, we deduce
\be\label{4.14}
\|a^\ell{\rm div}\bu\|_{\dot{B}_{2,\infty}^{-\sigma_1}}\lesssim \|{\rm div}\bu\|_{\dot{B}_{p,1}^{\frac{N}{p}}}
\|a\|_{\dot{B}_{2,\infty}^{-\sigma_1}}^\ell\lesssim \Big(\|\bu\|_{\dot{B}_{2,1}^{\frac{N}{2}+1}}^\ell+\|\bu\|_{\dot{B}_{p,1}^{\frac{N}{p}+1}}^h\Big)
\|a\|_{\dot{B}_{2,\infty}^{-\sigma_1}}^\ell
\ee
and
\be\label{4.15}
\|a^h{\rm div}\bu^\ell\|_{\dot{B}_{2,\infty}^{-\sigma_1}}\lesssim \|a^h\|_{\dot{B}_{p,1}^{\frac{N}{p}}}
\|{\rm div}\bu^\ell\|_{\dot{B}_{2,\infty}^{-\sigma_1}}\lesssim \|a\|_{\dot{B}_{p,1}^{\frac{N}{p}}}^h
\|\bu\|_{\dot{B}_{2,\infty}^{-\sigma_1}}^\ell.
\ee
By virtue of  \eqref{4.2}, one gets
\be\label{4.16}
\begin{split}
\|a^h{\rm div}\bu^h\|_{\dot{B}_{2,\infty}^{-\sigma_1}}^\ell&\lesssim \|a^h\|_{\dot{B}_{p,1}^{\frac{N}{p}-1}}
\Big(\|{\rm div}\bu^h\|_{\dot{B}_{p,\infty}^{-\sigma_1+\frac{N}{p}-\frac{N}{2}+1}}+\|{\rm div}\bu^h\|_{\dot{B}_{p,\infty}^{-\sigma_1+\frac{2N}{p}-N+1}}\Big)\\[1ex]
&\lesssim\|a\|_{\dot{B}_{p,1}^{\frac{N}{p}}}^h\|\bu\|_{\dot{B}_{p,1}^{\frac{N}{p}+1}}^h,
\end{split}
\ee
where we used that
$
-\sigma_1+\frac{2N}{p}-N+2\leq -\sigma_1+\frac{N}{p}-\frac{N}{2}+2<\frac{N}{p}+1
$
since $\sigma_1>1-\frac{N}{2}$ and $p\geq 2$.

\underline{Estimate of $\bu\cdot\nabla a$}. Decomposing $\bu\cdot\nabla a=\bu^\ell\cdot\nabla a^\ell+\bu^h\cdot\nabla a^\ell+\bu^\ell\cdot\nabla a^h+\bu^h\cdot\nabla a^h$, we deduce from \eqref{4.1} that
\be\label{4.17}
\|\bu^\ell\nabla a^\ell\|_{\dot{B}_{2,\infty}^{-\sigma_1}}\lesssim\|\nabla a^\ell\|_{\dot{B}_{p,1}^{\frac{N}{p}}}
\|\bu^\ell\|_{\dot{B}_{2,\infty}^{-\sigma_1}}\lesssim\|a\|_{\dot{B}_{2,1}^{\frac{N}{2}+1}}^\ell
\|\bu\|_{\dot{B}_{2,\infty}^{-\sigma_1}}^\ell,
\ee
and
\be\label{4.18}
\|\bu^h\nabla a^\ell\|_{\dot{B}_{2,\infty}^{-\sigma_1}}\lesssim\|\bu^h\|_{\dot{B}_{p,1}^{\frac{N}{p}}}
\|\nabla a^\ell\|_{\dot{B}_{2,\infty}^{-\sigma_1}}\lesssim\|\bu\|_{\dot{B}_{p,1}^{\frac{N}{p}+1}}^h
\|a\|_{\dot{B}_{2,\infty}^{-\sigma_1}}^\ell.
\ee
From \eqref{4.2}, one arrives at
\be\label{4.19}
\begin{split}
\|\bu^\ell\nabla a^h\|_{\dot{B}_{2,\infty}^{-\sigma_1}}^\ell&\lesssim\|\nabla a^h\|_{\dot{B}_{p,1}^{\frac{N}{p}-1}}
\Big(\|\bu^\ell\|_{\dot{B}_{p,\infty}^{-\sigma_1+\frac{N}{p}-\frac{N}{2}+1}}
+\|\bu^\ell\|_{\dot{B}_{p,\infty}^{-\sigma_1+\frac{2N}{p}-N+1}}\Big)\\[1ex]
&\lesssim\|a^h\|_{\dot{B}_{p,1}^{\frac{N}{p}}}\|\bu^\ell\|_{\dot{B}_{p,\infty}^{-\sigma_1+\frac{2N}{p}-N+1}}
\lesssim\|a\|_{\dot{B}_{p,1}^{\frac{N}{p}}}^h\|\bu\|_{\dot{B}_{2,\infty}^{-\sigma_1}}^\ell,
\end{split}
\ee
where we used that $-\sigma_1+\frac{2N}{p}-N+1\leq -\sigma_1+\frac{N}{p}-\frac{N}{2}+1$ in the second inequality
and $\|\bu\|_{\dot{B}_{2,\infty}^{-\sigma_1}}^\ell
\hookrightarrow\|\bu\|_{\dot{B}_{p,\infty}^{-\sigma_1+\frac{2N}{p}-N+1}}^\ell$ when $2\leq p\leq \frac{2N}{N-2}$ in the last inequality. For the term $\bu^h\nabla a^h$, also by \eqref{4.2}, it follows that
\be\label{4.20}
\begin{split}
\|\bu^h\nabla a^h\|_{\dot{B}_{2,\infty}^{-\sigma_1}}^\ell&\lesssim\|\nabla a^h\|_{\dot{B}_{p,1}^{\frac{N}{p}-1}}
\Big(\|\bu^h\|_{\dot{B}_{p,\infty}^{-\sigma_1+\frac{N}{p}-\frac{N}{2}+1}}
+\|\bu^h\|_{\dot{B}_{p,\infty}^{-\sigma_1+\frac{2N}{p}-N+1}}\Big)\\[1ex]
&\lesssim\|a^h\|_{\dot{B}_{p,1}^{\frac{N}{p}}}\|\bu^h\|_{\dot{B}_{p,\infty}^{-\sigma_1+\frac{N}{p}-\frac{N}{2}+1}}
\lesssim\|a\|_{\dot{B}_{p,1}^{\frac{N}{p}}}^h\|\bu\|_{\dot{B}_{p,1}^{\frac{N}{p}+1}}^h,
\end{split}
\ee
where we have applied that
$
-\sigma_1+\frac{2N}{p}-N+1\leq -\sigma_1+\frac{N}{p}-\frac{N}{2}+1\leq \frac{N}{p}+1,
$
since $\sigma_1> 1-\frac{N}{2}$ and $p\geq 2$.

In what follows, we estimate $\|{\bf g}\|_{\dot{B}_{2,\infty}^{-\sigma_1}}^\ell$. Recall that
\be\nonumber
\begin{split}
&{\bf g}\equ -\bu\cdot\nabla\bu-\pi_1(a)\mathcal{A}\bu-\pi_2(a)\nabla a+\frac{1}{1+a}{\rm div}\Big(2\widetilde\mu(a)
D(\bu)+\widetilde\lambda(a){\rm div}\bu\,{\rm Id}\Big)\\[1ex]
&\quad\quad+\pi_1(a)(\nabla(I\cdot\bH)-I\cdot\nabla\bH)-\frac{1}{1+a}\Big(\frac{1}{2}\nabla
|\bH|^2-\bH\cdot\nabla\bH\Big).
\end{split}
\ee

\underline{Estimate of $\bu\cdot\nabla\bu$}. Decompose $\bu\cdot\nabla\bu=\bu^\ell\cdot\nabla\bu^\ell+\bu^\ell\cdot\nabla\bu^h
+\bu^h\cdot\nabla\bu^\ell+\bu^h\cdot\nabla\bu^h$.
It holds from \eqref{4.1} that
\be\label{4.23}
\|\bu^\ell\cdot\nabla \bu^\ell\|_{\dot{B}_{2,\infty}^{-\sigma_1}}\lesssim\|\nabla \bu^\ell\|_{\dot{B}_{p,1}^{\frac{N}{p}}}
\|\bu^\ell\|_{\dot{B}_{2,\infty}^{-\sigma_1}}\lesssim\|\bu\|_{\dot{B}_{2,1}^{\frac{N}{2}+1}}^\ell
\|\bu\|_{\dot{B}_{2,\infty}^{-\sigma_1}}^\ell,
\ee
\be\label{4.24}
\|\bu^h\cdot\nabla \bu^\ell\|_{\dot{B}_{2,\infty}^{-\sigma_1}}\lesssim\|\bu^h\|_{\dot{B}_{p,1}^{\frac{N}{p}}}
\|\nabla\bu^\ell\|_{\dot{B}_{2,\infty}^{-\sigma_1}}\lesssim\|\bu\|_{\dot{B}_{p,1}^{\frac{N}{p}+1}}^h
\|\bu\|_{\dot{B}_{2,\infty}^{-\sigma_1}}^\ell.
\ee
In a similar way as deriving \eqref{4.19} and \eqref{4.20}, one has by \eqref{4.2} that
\be\label{4.124}
\begin{split}
\|\bu^\ell\cdot\nabla \bu^h\|_{\dot{B}_{2,\infty}^{-\sigma_1}}^\ell&\lesssim\|\nabla \bu^h\|_{\dot{B}_{p,1}^{\frac{N}{p}-1}}
\Big(\|\bu^\ell\|_{\dot{B}_{p,\infty}^{-\sigma_1+\frac{N}{p}-\frac{N}{2}+1}}
+\|\bu^\ell\|_{\dot{B}_{p,\infty}^{-\sigma_1+\frac{2N}{p}-N+1}}\Big)\\[1ex]
&\lesssim\|\bu^h\|_{\dot{B}_{p,1}^{\frac{N}{p}}}\|\bu^\ell\|_{\dot{B}_{p,\infty}^{-\sigma_1+\frac{2N}{p}-N+1}}
\lesssim\|\bu\|_{\dot{B}_{p,1}^{\frac{N}{p}+1}}^h\|\bu\|_{\dot{B}_{2,\infty}^{-\sigma_1}}^\ell,
\end{split}
\ee
and
\be\label{4.25}
\begin{split}
\|\bu^h\cdot\nabla \bu^h\|_{\dot{B}_{2,\infty}^{-\sigma_1}}^\ell&\lesssim\|\nabla \bu^h\|_{\dot{B}_{p,1}^{\frac{N}{p}-1}}
\Big(\|\bu^h\|_{\dot{B}_{p,\infty}^{-\sigma_1+\frac{N}{p}-\frac{N}{2}+1}}
+\|\bu^h\|_{\dot{B}_{p,\infty}^{-\sigma_1+\frac{2N}{p}-N+1}}\Big)\\[1ex]
&\lesssim\|\bu^h\|_{\dot{B}_{p,1}^{\frac{N}{p}}}\|\bu^h\|_{\dot{B}_{p,\infty}^{-\sigma_1+\frac{N}{p}-\frac{N}{2}+1}}
\lesssim\|\bu\|_{\dot{B}_{p,1}^{\frac{N}{p}}}^h\|\bu\|_{\dot{B}_{p,1}^{\frac{N}{p}}}^h.
\end{split}
\ee

\underline{Estimate of $\pi_1(a)\mathcal{A}\bu$}. Keeping in mind that $\pi_1(0)=0$, one may write
$$
\pi_1(a)=\pi_1^\prime(0)a+\bar{\pi}_1(a)a
$$
for some smooth function $\bar{\pi}_1$ vanishing at $0$. Thus, through \eqref{4.1} again, we have
\be\label{4.26}
\|a^\ell\mathcal{A}\bu^\ell\|_{\dot{B}_{2,\infty}^{-\sigma_1}}\lesssim\|\mathcal{A}\bu^\ell\|_{\dot{B}_{p,1}^{\frac{N}{p}}}
\|a^\ell\|_{\dot{B}_{2,\infty}^{-\sigma_1}}\lesssim\|\bu\|_{\dot{B}_{2,1}^{\frac{N}{2}+1}}^\ell
\|a\|_{\dot{B}_{2,\infty}^{-\sigma_1}}^\ell,
\ee
and
\be\label{4.27}
\|a^h\mathcal{A}\bu^\ell\|_{\dot{B}_{2,\infty}^{-\sigma_1}}\lesssim\|a^h\|_{\dot{B}_{p,1}^{\frac{N}{p}}}
\|\mathcal{A}\bu^\ell\|_{\dot{B}_{2,\infty}^{-\sigma_1}}\lesssim\|a\|_{\dot{B}_{p,1}^{\frac{N}{p}}}^h
\|\bu\|_{\dot{B}_{2,\infty}^{-\sigma_1}}^\ell.
\ee
Arguing similarly as \eqref{4.19} and \eqref{4.20}, one has
\be\label{4.28}
\begin{split}
\|a^\ell\mathcal{A} \bu^h\|_{\dot{B}_{2,\infty}^{-\sigma_1}}^\ell&\lesssim\|\mathcal{A} \bu^h\|_{\dot{B}_{p,1}^{\frac{N}{p}-1}}
\Big(\|a^\ell\|_{\dot{B}_{p,\infty}^{-\sigma_1+\frac{N}{p}-\frac{N}{2}+1}}
+\|a^\ell\|_{\dot{B}_{p,\infty}^{-\sigma_1+\frac{2N}{p}-N+1}}\Big)\\[1ex]
&\lesssim\|\bu\|_{\dot{B}_{p,1}^{\frac{N}{p}+1}}^h\|a^\ell\|_{\dot{B}_{p,\infty}^{-\sigma_1+\frac{2N}{p}-N+1}}
\lesssim\|\bu\|_{\dot{B}_{p,1}^{\frac{N}{p}+1}}^h\|a\|_{\dot{B}_{2,\infty}^{-\sigma_1}}^\ell,
\end{split}
\ee
and
\be\label{4.29}
\begin{split}
\|a^h\mathcal{A}\bu^h\|_{\dot{B}_{2,\infty}^{-\sigma_1}}^\ell&\lesssim\|\mathcal{A}\bu^h\|_{\dot{B}_{p,1}^{\frac{N}{p}-1}}
\Big(\|a^h\|_{\dot{B}_{p,\infty}^{-\sigma_1+\frac{N}{p}-\frac{N}{2}+1}}
+\|a^h\|_{\dot{B}_{p,\infty}^{-\sigma_1+\frac{2N}{p}-N+1}}\Big)\\[1ex]
&\lesssim\|\bu^h\|_{\dot{B}_{p,1}^{\frac{N}{p}+1}}\|a^h\|_{\dot{B}_{p,\infty}^{-\sigma_1+\frac{N}{p}-\frac{N}{2}+1}}
\lesssim\|\bu\|_{\dot{B}_{p,1}^{\frac{N}{p}+1}}^h\|a\|_{\dot{B}_{p,1}^{\frac{N}{p}}}^h.
\end{split}
\ee
On the other hand, from \eqref{4.1}, \eqref{4.2}, Proposition \ref{pra.4} and Corollaries \ref{co2.1} and \ref{co2.2}, we have
\be\label{4.30}
\|\bar{\pi}_1(a)a\mathcal{A}\bu^\ell\|_{\dot{B}_{2,\infty}^{-\sigma_1}}
\lesssim\|\bar{\pi}_1(a)a\|_{\dot{B}_{p,1}^{\frac{N}{p}}}\|\mathcal{A}\bu^\ell\|_{\dot{B}_{2,\infty}^{-\sigma_1}}
\lesssim\|a\|_{\dot{B}_{p,1}^{\frac{N}{p}}}^2\|\bu\|_{\dot{B}_{2,\infty}^{-\sigma_1}}^\ell,
\ee
and
\be\label{4.31}
\begin{split}
\|\bar{\pi}_1(a)a\mathcal{A}\bu^h\|_{\dot{B}_{2,\infty}^{-\sigma_1}}^\ell
&\lesssim\|\mathcal{A}\bu^h\|_{\dot{B}_{p,1}^{\frac{N}{p}-1}}
\Big(\|\bar{\pi}_1(a)a\|_{\dot{B}_{p,\infty}^{-\sigma_1+\frac{N}{p}-\frac{N}{2}+1}}
+\|\bar{\pi}_1(a)a\|_{\dot{B}_{p,\infty}^{-\sigma_1+\frac{2N}{p}-N+1}}\Big)\\[1ex]
&\lesssim\|\bu^h\|_{\dot{B}_{p,1}^{\frac{N}{p}+1}}\|\bar{\pi}_1(a)\|_{\dot{B}_{p,1}^{\frac{N}{p}}}
\Big(\|a\|_{\dot{B}_{p,\infty}^{-\sigma_1+\frac{N}{p}-\frac{N}{2}+1}}
+\|a\|_{\dot{B}_{p,\infty}^{-\sigma_1+\frac{2N}{p}-N+1}}\Big)\\[1ex]
&\lesssim\|\bu\|_{\dot{B}_{p,1}^{\frac{N}{p}+1}}^h
\|a\|_{\dot{B}_{p,1}^{\frac{N}{p}}}\Big(\|a\|_{\dot{B}_{p,\infty}^{-\sigma_1+\frac{N}{p}-\frac{N}{2}+1}}^h
+\|a\|_{\dot{B}_{p,\infty}^{-\sigma_1+\frac{2N}{p}-N+1}}^\ell\Big)\\[1ex]
&\lesssim\|\bu\|_{\dot{B}_{p,1}^{\frac{N}{p}+1}}^h
\|a\|_{\dot{B}_{p,1}^{\frac{N}{p}}}\Big(\|a\|_{\dot{B}_{p,1}^{\frac{N}{p}}}^h
+\|a\|_{\dot{B}_{2,\infty}^{-\sigma_1}}^\ell\Big).
\end{split}
\ee

\underline{Estimate of $\pi_2(a)\nabla a$}. In view of $\pi_2(0)=0$, we may write $\pi_2(a)=\pi_2^\prime(0)a+\bar{\pi}_2(a)a$, here $\bar{\pi}_2$ is a smooth function fulfilling $\bar{\pi}_2(0)=0$. For the term $a\nabla a$, we have
\be\label{4.32}
\|a^\ell\nabla a^\ell\|_{\dot{B}_{2,\infty}^{-\sigma_1}}\lesssim\|\nabla a^\ell\|_{\dot{B}_{p,1}^{\frac{N}{p}}}
\|a^\ell\|_{\dot{B}_{2,\infty}^{-\sigma_1}}\lesssim\|a\|_{\dot{B}_{2,1}^{\frac{N}{2}+1}}^\ell
\|a\|_{\dot{B}_{2,\infty}^{-\sigma_1}}^\ell,
\ee
and
\be\label{4.33}
\|a^h\nabla a^\ell\|_{\dot{B}_{2,\infty}^{-\sigma_1}}\lesssim\|a^h\|_{\dot{B}_{p,1}^{\frac{N}{p}}}
\|\nabla a^\ell\|_{\dot{B}_{2,\infty}^{-\sigma_1}}\lesssim\|a\|_{\dot{B}_{p,1}^{\frac{N}{p}}}^h
\|a\|_{\dot{B}_{2,\infty}^{-\sigma_1}}^\ell.
\ee
Arguing similarly as \eqref{4.28} and \eqref{4.29}, one has
\be\label{4.34}
\begin{split}
\|a^\ell\nabla a^h\|_{\dot{B}_{2,\infty}^{-\sigma_1}}^\ell&\lesssim\|\nabla a^h\|_{\dot{B}_{p,1}^{\frac{N}{p}-1}}
\Big(\|a^\ell\|_{\dot{B}_{p,\infty}^{-\sigma_1+\frac{N}{p}-\frac{N}{2}+1}}
+\|a^\ell\|_{\dot{B}_{p,\infty}^{-\sigma_1+\frac{2N}{p}-N+1}}\Big)\\[1ex]
&\lesssim\|a\|_{\dot{B}_{p,1}^{\frac{N}{p}}}^h\|a^\ell\|_{\dot{B}_{p,\infty}^{-\sigma_1+\frac{2N}{p}-N+1}}
\lesssim\|a\|_{\dot{B}_{p,1}^{\frac{N}{p}}}^h\|a\|_{\dot{B}_{2,\infty}^{-\sigma_1}}^\ell,
\end{split}
\ee
and
\be\label{4.35}
\begin{split}
\|a^h\nabla a^h\|_{\dot{B}_{2,\infty}^{-\sigma_1}}^\ell&\lesssim\|\nabla a^h\|_{\dot{B}_{p,1}^{\frac{N}{p}-1}}
\Big(\|a^h\|_{\dot{B}_{p,\infty}^{-\sigma_1+\frac{N}{p}-\frac{N}{2}+1}}
+\|a^h\|_{\dot{B}_{p,\infty}^{-\sigma_1+\frac{2N}{p}-N+1}}\Big)\\[1ex]
&\lesssim\|a^h\|_{\dot{B}_{p,1}^{\frac{N}{p}}}\|a^h\|_{\dot{B}_{p,\infty}^{-\sigma_1+\frac{N}{p}-\frac{N}{2}+1}}
\lesssim\|a\|_{\dot{B}_{p,1}^{\frac{N}{p}}}^h\|a\|_{\dot{B}_{p,1}^{\frac{N}{p}}}^h.
\end{split}
\ee
As for the term $\bar{\pi}_2(a)a\nabla a$, we use the decomposition $\bar{\pi}_2(a)a\nabla a=\bar{\pi}_2(a)a\nabla a^\ell+\bar{\pi}_2(a)a\nabla a^h$ and get from \eqref{4.1}-\eqref{4.2}, Corollary \ref{co2.2} and Proposition \ref{pra.4} again that
\be\label{4.36}
\begin{split}
\|\bar{\pi}_2(a)a\nabla a^\ell\|_{\dot{B}_{2,\infty}^{-\sigma_1}}\lesssim\|\bar{\pi}_2(a)a\|_{\dot{B}_{p,1}^{\frac{N}{p}}}
\|\nabla a\|_{\dot{B}_{2,\infty}^{-\sigma_1}}^\ell\lesssim\|a\|_{\dot{B}_{p,1}^{\frac{N}{p}}}^2
\|a\|_{\dot{B}_{2,\infty}^{-\sigma_1}}^\ell,
\end{split}
\ee
and
\be\label{4.37}
\begin{split}
\|\bar{\pi}_2(a)a\nabla a^h\|_{\dot{B}_{2,\infty}^{-\sigma_1}}^\ell&\lesssim\|\nabla a^h\|_{\dot{B}_{p,1}^{\frac{N}{p}-1}}
\Big(\|\bar{\pi}_2(a)a\|_{\dot{B}_{p,\infty}^{-\sigma_1+\frac{N}{p}-\frac{N}{2}+1}}
+\|\bar{\pi}_2(a)a\|_{\dot{B}_{p,\infty}^{-\sigma_1+\frac{2N}{p}-N+1}}\Big)\\[1ex]
&\lesssim\|a^h\|_{\dot{B}_{p,1}^{\frac{N}{p}}}\|\bar{\pi}_2(a)\|_{\dot{B}_{p,1}^{\frac{N}{p}}}
\Big(\|a\|_{\dot{B}_{p,\infty}^{-\sigma_1+\frac{N}{p}-\frac{N}{2}+1}}
+\|a\|_{\dot{B}_{p,\infty}^{-\sigma_1+\frac{2N}{p}-N+1}}\Big)\\[1ex]
&\lesssim\|a\|_{\dot{B}_{p,1}^{\frac{N}{p}}}^2\Big(\|a\|_{\dot{B}_{2,\infty}^{-\sigma_1}}^\ell
+\|a\|_{\dot{B}_{p,1}^{\frac{N}{p}}}^h\Big).
\end{split}
\ee

\underline{Estimate of $\frac{1}{1+a}(2\widetilde{\mu}(a){\rm div}D(\bu)+\widetilde{\lambda}(a)\nabla{\rm div}\bu)$}.
The estimate of this term could be similarly handled  as the term $\pi_1(a)\mathcal{A}\bu$ and we omit it here.

\underline{Estimate of $\frac{1}{1+a}(2\widetilde{\mu}^\prime(a)D(\bu)\cdot\nabla a+\widetilde{\lambda}^\prime(a){\rm div}\bu\nabla a)$}. We only deal with the term $\frac{2\widetilde{\mu}^\prime(a)}{1+a}D(\bu)\cdot\nabla a$ and the remainder term could be  similarly handled. Denote by $J(a)$ the smooth function fulfilling
\be\label{4.1000}
J^\prime(a)=\frac{2\mu^\prime(a)}{1+a}
\,\,\,{\rm and}\,\,\, J(0)=0, \,\,\,{\rm so\,\, that}\,\,\, \nabla J(a)=\frac{2\mu^\prime(a)}{1+a}\nabla a.
\ee
Decomposing $J(a)=J^\prime(0)a+\bar{J}(a)a$ implies
$\nabla J(a)=J^\prime(0)\nabla a+\nabla (\bar{J}(a)a)$. Then we have from \eqref{4.1} and \eqref{4.2} that
\be\label{4.38}
\|\nabla a^\ell D(\bu)^\ell\|_{\dot{B}_{2,\infty}^{-\sigma_1}}\lesssim\|\nabla a^\ell\|_{\dot{B}_{p,1}^{\frac{N}{p}}}
\|D(\bu)^\ell\|_{\dot{B}_{2,\infty}^{-\sigma_1}}\lesssim\|a\|_{\dot{B}_{2,1}^{\frac{N}{2}+1}}^\ell
\|\bu\|_{\dot{B}_{2,\infty}^{-\sigma_1}}^\ell,
\ee
\be\label{4.39}
\|\nabla a^\ell D(\bu)^h\|_{\dot{B}_{2,\infty}^{-\sigma_1}}\lesssim\|D(\bu)^h\|_{\dot{B}_{p,1}^{\frac{N}{p}}}
\|\nabla a^\ell\|_{\dot{B}_{2,\infty}^{-\sigma_1}}\lesssim\|\bu\|_{\dot{B}_{p,1}^{\frac{N}{p}+1}}^h
\|a\|_{\dot{B}_{2,\infty}^{-\sigma_1}}^\ell,
\ee
and
\be\label{4.40}
\begin{split}
\|\nabla a^h D(\bu)\|_{\dot{B}_{2,\infty}^{-\sigma_1}}^\ell&\lesssim\|\nabla a^h\|_{\dot{B}_{p,1}^{\frac{N}{p}-1}}
\Big(\|D(\bu)\|_{\dot{B}_{p,\infty}^{-\sigma_1+\frac{N}{p}-\frac{N}{2}+1}}
+\|D(\bu)\|_{\dot{B}_{p,\infty}^{-\sigma_1+\frac{2N}{p}-N+1}}\Big)\\[1ex]
&\lesssim\|a^h\|_{\dot{B}_{p,1}^{\frac{N}{p}}}
\Big(\|\bu\|_{\dot{B}_{p,\infty}^{-\sigma_1+\frac{N}{p}-\frac{N}{2}+2}}^h
+\|\bu\|_{\dot{B}_{p,\infty}^{-\sigma_1+\frac{2N}{p}-N+2}}^\ell\Big)\\[1ex]
&\lesssim\|a\|_{\dot{B}_{p,1}^{\frac{N}{p}}}^h\Big(\|\bu\|_{\dot{B}_{p,1}^{\frac{N}{p}+1}}^h+\|\bu\|_{\dot{B}_{2,\infty}^{-\sigma_1}}^\ell
\Big).
\end{split}
\ee

\underline{Estimate of $\pi_1(a)(\nabla (I\cdot \bH)-I\cdot \nabla \bH)$}. Similar as above, we decompose
$\pi_1(a)=\pi_1^\prime(0)a+\bar{\pi}_1(a)a$. Firstly, the estimate of $a(\nabla (I\cdot \bH)-I\cdot \nabla \bH)$
is similar to that of  $a{\rm div}\bu$ and we omit it here. The remaining term may be estimated as follows.
\be\label{4.41}
\|\bar{\pi}_1(a) a(\nabla(I\cdot\bH)-I\cdot\nabla\bH)^\ell\|_{\dot{B}_{2,\infty}^{-\sigma_1}}\lesssim\|\bar{\pi}_1(a) a\|_{\dot{B}_{p,1}^{\frac{N}{p}}}
\|\nabla \bH^\ell\|_{\dot{B}_{2,\infty}^{-\sigma_1}}\lesssim\|a\|_{\dot{B}_{p,1}^{\frac{N}{p}}}^2
\|\bH\|_{\dot{B}_{2,\infty}^{-\sigma_1}}^\ell,
\ee
and
\be\label{4.42}
\begin{split}
&\quad\|\bar{\pi}_1(a) a(\nabla(I\cdot\bH)-I\cdot\nabla\bH)^h\|_{\dot{B}_{2,\infty}^{-\sigma_1}}^\ell\\[1ex]
&\lesssim\|\nabla \bH^h\|_{\dot{B}_{p,1}^{\frac{N}{p}-1}}
\Big(\|\bar{\pi}_1(a)a\|_{\dot{B}_{p,\infty}^{-\sigma_1+\frac{N}{p}-\frac{N}{2}+1}}
+\|\bar{\pi}_1(a)a\|_{\dot{B}_{p,\infty}^{-\sigma_1+\frac{2N}{p}-N+1}}\Big)\\[1ex]
&\lesssim\|\bH^h\|_{\dot{B}_{p,1}^{\frac{N}{p}+1}}\|\bar{\pi}_1(a)\|_{\dot{B}_{p,1}^{\frac{N}{p}}}
\Big(\|a\|_{\dot{B}_{p,\infty}^{-\sigma_1+\frac{N}{p}-\frac{N}{2}+1}}
+\|a\|_{\dot{B}_{p,\infty}^{-\sigma_1+\frac{2N}{p}-N+1}}\Big)\\[1ex]
&\lesssim\|\bH\|_{\dot{B}_{p,1}^{\frac{N}{p}+1}}^h\|a\|_{\dot{B}_{p,1}^{\frac{N}{p}}}\Big(\|a\|_{\dot{B}_{2,\infty}^{-\sigma_1}}^\ell
+\|a\|_{\dot{B}_{p,1}^{\frac{N}{p}}}^h\Big).
\end{split}
\ee

\underline{Estimate of $\frac{1}{1+a}(\frac{1}{2}\nabla |\bH|^2-\bH\cdot \nabla \bH)$}. Since $\frac{1}{1+a}=1-\pi_1(a)$, it follows that
\be\label{4.200}
\begin{split}
\frac{1}{1+a}\Big(\frac{1}{2}\nabla |\bH|^2-\bH\cdot \nabla \bH\Big)&=\Big(\frac{1}{2}\nabla |\bH|^2-\bH\cdot \nabla \bH\Big)
-\pi_1(a)\Big(\frac{1}{2}\nabla |\bH|^2-\bH\cdot \nabla \bH\Big)\\[1ex]
&=\bH\cdot((\nabla \bH)^T-\nabla \bH)
-\pi_1(a)\bH\cdot((\nabla \bH)^T-\nabla \bH),
\end{split}
\ee
where the superscript $T$ represents  the transpose of a matrix.

For the term with $\bH\cdot((\nabla \bH)^T-\nabla \bH)$, we can handle it similar to the term $\bu\cdot\nabla\bu$, while regarding the term with $\pi_1(a)\bH\cdot((\nabla \bH)^T-\nabla \bH)$, we have from \eqref{4.1} and \eqref{4.2} again that
\be\label{4.43}
\|\pi_1(a)\bH\cdot(\nabla\bH)^\ell\|_{\dot{B}_{2,\infty}^{-\sigma_1}}\lesssim\|\pi_1(a)\bH\|_{\dot{B}_{p,1}^{\frac{N}{p}}}
\|\nabla\bH\|_{\dot{B}_{2,\infty}^{-\sigma_1}}^\ell\lesssim \|a\|_{\dot{B}_{p,1}^{\frac{N}{p}}}\|\bH\|_{\dot{B}_{p,1}^{\frac{N}{p}}}\|\bH\|_{\dot{B}_{2,\infty}^{-\sigma_1}}^\ell,
\ee
and
\be\label{4.44}
\begin{split}
\|\pi_1(a)\bH\cdot(\nabla\bH)^h\|_{\dot{B}_{2,\infty}^{-\sigma_1}}^\ell
&\lesssim\|\nabla \bH^h\|_{\dot{B}_{p,1}^{\frac{N}{p}-1}}
\Big(\|{\pi}_1(a)\bH\|_{\dot{B}_{p,\infty}^{-\sigma_1+\frac{N}{p}-\frac{N}{2}+1}}
+\|{\pi}_1(a)\bH\|_{\dot{B}_{p,\infty}^{-\sigma_1+\frac{2N}{p}-N+1}}\Big)\\[1ex]
&\lesssim\|\bH^h\|_{\dot{B}_{p,1}^{\frac{N}{p}}}\|{\pi}_1(a)\|_{\dot{B}_{p,1}^{\frac{N}{p}}}
\Big(\|\bH\|_{\dot{B}_{p,\infty}^{-\sigma_1+\frac{N}{p}-\frac{N}{2}+1}}^h
+\|\bH\|_{\dot{B}_{p,\infty}^{-\sigma_1+\frac{2N}{p}-N+1}}^\ell\Big)\\[1ex]
&\lesssim\|\bH\|_{\dot{B}_{p,1}^{\frac{N}{p}}}^h\|a\|_{\dot{B}_{p,1}^{\frac{N}{p}}}\Big(
\|\bH\|_{\dot{B}_{p,1}^{\frac{N}{p}}}^h+\|\bH\|_{\dot{B}_{2,\infty}^{-\sigma_1}}^\ell\Big).
\end{split}
\ee

\underline{Estimate of ${\bf m}$}.
Since
$${\bf m}\equ -\bH({\rm div}\bu)+\bH\cdot\nabla\bu-\bu\cdot\nabla\bH,$$
 then its estimation is  similar to that of  $\bu\cdot\nabla \bu$ and we omit it here. Finally, inserting all  estimates above into \eqref{4.13}, we complete the proof of \eqref{4.11}.
\end{proof}

By the definition of $\mathcal{X}_p(t)$ in Theorem \ref{th1.1}, one has
\be\nonumber
\begin{split}
\|(a^\ell, \bu^\ell, \bH^\ell)\|_{L_t^2(\dot{B}_{p,1}^{\frac{N}{p}})}&\lesssim\|(a^\ell, \bu^\ell, \bH^\ell)\|_{L_t^\infty(\dot{B}_{p,1}^{\frac{N}{p}-1})}^{\frac{1}{2}}\|(a^\ell, \bu^\ell, \bH^\ell)\|_{L_t^1(\dot{B}_{p,1}^{\frac{N}{p}+1})}^{\frac{1}{2}}\\[1ex]
&\lesssim\Big(\|(a, \bu, \bH)\|_{L_t^\infty(\dot{B}_{2,1}^{\frac{N}{2}-1})}^\ell\Big)^{\frac{1}{2}}\Big(\|(a, \bu, \bH)\|_{L_t^1(\dot{B}_{2,1}^{\frac{N}{2}+1})}^\ell\Big)^{\frac{1}{2}},
\end{split}
\ee
\be\nonumber
\|a^h\|_{L_t^2(\dot{B}_{p,1}^{\frac{N}{p}})}\lesssim\Big(\|a\|_{L_t^\infty(\dot{B}_{p,1}^{\frac{N}{p}})}^h\Big)^{\frac{1}{2}}
\Big(\|a\|_{L_t^1(\dot{B}_{p,1}^{\frac{N}{p}})}^h\Big)^{\frac{1}{2}},
\ee
and
\be\nonumber
\|(\bu^h,\bH^h)\|_{L_t^2(\dot{B}_{p,1}^{\frac{N}{p}})}\lesssim\Big(\|(\bu,\bH)\|_{L_t^\infty(\dot{B}_{p,1}^{\frac{N}{p}-1})}^h\Big)^{\frac{1}{2}}
\Big(\|(\bu,\bH)\|_{L_t^1(\dot{B}_{p,1}^{\frac{N}{p}+1})}^h\Big)^{\frac{1}{2}}.
\ee
On the other hand, it follows that
\be\nonumber
\|a\|_{L_t^\infty(\dot{B}_{p,1}^{\frac{N}{p}})}\lesssim\|a\|_{L_t^\infty(\dot{B}_{p,1}^{\frac{N}{p}})}^\ell
+\|a\|_{L_t^\infty(\dot{B}_{p,1}^{\frac{N}{p}})}^h\lesssim\|a\|_{L_t^\infty(\dot{B}_{2,1}^{\frac{N}{2}-1})}^\ell
+\|a\|_{L_t^\infty(\dot{B}_{p,1}^{\frac{N}{p}})}^h.
\ee
Then, we have
\be\label{4.45}
\int_0^t(A_1(\tau)+A_2(\tau))d\tau\leq \mathcal{X}_p+\mathcal{X}_p^2+\mathcal{X}_p^3\leq C\mathcal{X}_{p,0},
\ee
which yields from Gronwall's inequality (see for example, Page 360 of \cite{mitrinovic1994inequalities}) that
\be\label{4.46}
\|(a,\bu,\bH)\|_{\dot{B}_{2,\infty}^{-\sigma_1}}^\ell\leq C_0
\ee
for all $t\geq 0$, where $C_0>0$ depends on the norm $\|(a_0, \bu_0, \bH_0)\|_{\dot{B}_{2,\infty}^{-\sigma_1}}^\ell$.
\section{Proofs of main results}\label{s:6}
This section is devoted to proving Theorem \ref{th1.2} and Corollary \ref{col1}.
\subsection{Proof of Theorem \ref{th1.2}}
From Lemmas \ref{le1} and \ref{le2}, one deduces that
\be\label{5.1}
\begin{split}
  &\quad\frac{d}{dt}\Big(\|(a, \bu, \bH)\|_{\dot{B}_{2,1}^{\frac{N}{2}-1}}^\ell+\|(\nabla a, \bu, \bH)\|_{\dot{B}_{p,1}^{\frac{N}{p}-1}}^h\Big)\\[1ex]
  &\quad\quad\quad\quad\quad\quad\quad\quad\quad\quad\quad\quad\quad\quad
  +\Big(\|(a, \bu, \bH)\|_{\dot{B}_{2,1}^{\frac{N}{2}+1}}^\ell
  +\|a\|_{\dot{B}_{p,1}^{\frac{N}{p}}}^h+\|(\bu, \bH)\|_{\dot{B}_{p,1}^{\frac{N}{p}+1}}^h\Big)\\[1ex]
  &\lesssim\|(f, {\bf g}, {\bf m})\|_{\dot{B}_{2,1}^{\frac{N}{2}-1}}^\ell+\|f\|_{\dot{B}_{p,1}^{\frac{N}{p}-2}}^h
  +\|({\bf g}, {\bf m})\|_{\dot{B}_{p,1}^{\frac{N}{p}-1}}^h+\|\nabla\bu\|_{\dot{B}_{p,1}^{\frac{N}{p}}}\|a\|_{\dot{B}_{p,1}^{\frac{N}{p}}}.
\end{split}
\ee
In what follows, we deal with the terms in the right hand of \eqref{5.1} one by one. Firstly, for the last term, we have
\be\label{5.2}
\begin{split}
\|\nabla\bu\|_{\dot{B}_{p,1}^{\frac{N}{p}}}\|a\|_{\dot{B}_{p,1}^{\frac{N}{p}}}
&\lesssim\Big(\|a\|_{\dot{B}_{2,1}^{\frac{N}{2}-1}}^\ell+\|a\|_{\dot{B}_{p,1}^{\frac{N}{p}}}^h\Big)
\Big(\|\bu\|_{\dot{B}_{2,1}^{\frac{N}{2}+1}}^\ell+\|\bu\|_{\dot{B}_{p,1}^{\frac{N}{p}+1}}^h\Big)\\[1ex]
&\lesssim\mathcal{X}_p(t)\Big(\|\bu\|_{\dot{B}_{2,1}^{\frac{N}{2}+1}}^\ell+\|\bu\|_{\dot{B}_{p,1}^{\frac{N}{p}+1}}^h\Big).
\end{split}
\ee
Next, notice that
$$
\|f\|_{\dot{B}_{p,1}^{\frac{N}{p}-2}}^h\lesssim \|a\bu\|_{\dot{B}_{p,1}^{\frac{N}{p}-1}}^h.
$$
Decomposing $a\bu=a^\ell\bu^\ell+a^\ell\bu^h+a^h\bu$, we have
$$
\|a^\ell\bu^h\|_{\dot{B}_{p,1}^{\frac{N}{p}-1}}^h\lesssim\|a^\ell\|_{\dot{B}_{p,1}^{\frac{N}{p}-1}}
\|\bu^h\|_{\dot{B}_{p,1}^{\frac{N}{p}}}\lesssim\|a\|_{\dot{B}_{2,1}^{\frac{N}{2}-1}}^\ell
\|\bu\|_{\dot{B}_{p,1}^{\frac{N}{p}+1}}^h\lesssim\mathcal{X}_p(t)\|\bu\|_{\dot{B}_{p,1}^{\frac{N}{p}+1}}^h,
$$
and
$$
\|a^h\bu\|_{\dot{B}_{p,1}^{\frac{N}{p}-1}}^h\lesssim\|a^h\|_{\dot{B}_{p,1}^{\frac{N}{p}}}
\|\bu\|_{\dot{B}_{p,1}^{\frac{N}{p}-1}}\lesssim\mathcal{X}_p(t)\|a\|_{\dot{B}_{p,1}^{\frac{N}{p}}}^h.
$$
It follows from Corollary \ref{co2.1} and Bernstein inequality that
$$
\|a^\ell\bu^\ell\|_{\dot{B}_{p,1}^{\frac{N}{p}-1}}^h\lesssim\|a^\ell\bu^\ell\|_{\dot{B}_{2,1}^{\frac{N}{2}+1}}
\lesssim\|a^\ell\|_{L^\infty}\|\bu^\ell\|_{\dot{B}_{2,1}^{\frac{N}{2}+1}}
+\|\bu^\ell\|_{L^\infty}\|a^\ell\|_{\dot{B}_{2,1}^{\frac{N}{2}+1}}
\lesssim\mathcal{X}_p(t)\|(a,\bu)\|_{\dot{B}_{2,1}^{\frac{N}{2}+1}}^\ell.
$$
Therefore, we conclude that
\be\label{5.3}
\|f\|_{\dot{B}_{p,1}^{\frac{N}{p}-2}}^h\lesssim\mathcal{X}_p(t)\Big(\|a\|_{\dot{B}_{p,1}^{\frac{N}{p}}}^h
+\|(a,\bu)\|_{\dot{B}_{2,1}^{\frac{N}{2}+1}}^\ell+\|\bu\|_{\dot{B}_{p,1}^{\frac{N}{p}+1}}^h\Big).
\ee

Now we are in a position to bound $\|({\bf g}, {\bf m})\|_{\dot{B}_{p,1}^{\frac{N}{p}-1}}^h$ and the tools  are mainly involved with Corollaries \ref{co2.1} and \ref{co2.2}, Proposition \ref{pra.4} and Bernstein inequality.
\be\label{5.4}
\|\bu\cdot\nabla \bu\|_{\dot{B}_{p,1}^{\frac{N}{p}-1}}^h\lesssim\|\bu\|_{\dot{B}_{p,1}^{\frac{N}{p}-1}}
\|\nabla\bu\|_{\dot{B}_{p,1}^{\frac{N}{p}}}
\lesssim\mathcal{X}_p(t)\Big(\|\bu\|_{\dot{B}_{2,1}^{\frac{N}{2}+1}}^\ell+\|\bu\|_{\dot{B}_{p,1}^{\frac{N}{p}+1}}^h\Big).
\ee
\be\label{5.5}
\begin{split}
\|\pi_1(a)\mathcal{A}\bu\|_{\dot{B}_{p,1}^{\frac{N}{p}-1}}^h&\lesssim\|\pi_1(a)\|_{\dot{B}_{p,1}^{\frac{N}{p}}}
\|\mathcal{A}\bu\|_{\dot{B}_{p,1}^{\frac{N}{p}-1}}\\[1ex]
&
\lesssim\|a\|_{\dot{B}_{p,1}^{\frac{N}{p}}}\Big(\|\bu\|_{\dot{B}_{2,1}^{\frac{N}{2}+1}}^\ell
+\|\bu\|_{\dot{B}_{p,1}^{\frac{N}{p}+1}}^h\Big)
\lesssim\mathcal{X}_p(t)\Big(\|\bu\|_{\dot{B}_{2,1}^{\frac{N}{2}+1}}^\ell+\|\bu\|_{\dot{B}_{p,1}^{\frac{N}{p}+1}}^h\Big).
\end{split}
\ee
\be\label{5.6}
\begin{split}
&\|\pi_2(a)\nabla a\|_{\dot{B}_{p,1}^{\frac{N}{p}-1}}^h\lesssim\|\pi_2(a)\|_{\dot{B}_{p,1}^{\frac{N}{p}}}
\|\nabla a\|_{\dot{B}_{p,1}^{\frac{N}{p}-1}}\lesssim\|a\|_{\dot{B}_{p,1}^{\frac{N}{p}}}^2
\lesssim\|a^\ell\|_{\dot{B}_{p,1}^{\frac{N}{p}}}^2
+\|a^h\|_{\dot{B}_{p,1}^{\frac{N}{p}}}^2\\[1ex]
&\quad\quad\quad\quad\lesssim\|a^\ell\|_{\dot{B}_{p,1}^{\frac{N}{p}-1}}\|a^\ell\|_{\dot{B}_{p,1}^{\frac{N}{p}+1}}
+\|a^h\|_{\dot{B}_{p,1}^{\frac{N}{p}}}\|a^h\|_{\dot{B}_{p,1}^{\frac{N}{p}}}
\lesssim\mathcal{X}_p(t)\Big(\|a\|_{\dot{B}_{2,1}^{\frac{N}{2}+1}}^\ell+\|a\|_{\dot{B}_{p,1}^{\frac{N}{p}}}^h\Big).
\end{split}
\ee

The term with $\frac{1}{1+a}\Big(2\widetilde{\mu}(a){\rm div}D(\bu)+\widetilde{\lambda}(a)\nabla{\rm div}\bu\Big)$
could be similarly handled as the term $\pi_1(a)\mathcal{A}\bu$ and we omit it here.

Regarding the term with $\frac{1}{1+a}(2\widetilde{\mu}^\prime(a)D(\bu)\cdot\nabla a+\widetilde{\lambda}^\prime(a){\rm div}\bu\nabla a)$, as before we only perform the term $\frac{2\widetilde{\mu}^\prime(a)}{1+a}D(\bu)\cdot\nabla a$ and the other could be  handled similarly. Denote by $J(a)$ the smooth function fulfilling $J^\prime(a)=\frac{2\mu^\prime(a)}{1+a}$
and $J(0)=0$, so that $\nabla J(a)=\frac{2\mu^\prime(a)}{1+a}\nabla a$. Then we have
\be\label{5.7}
\begin{split}
&\|\nabla J(a)D(\bu)\|_{\dot{B}_{p,1}^{\frac{N}{p}-1}}^h\lesssim\|D(\bu)\|_{\dot{B}_{p,1}^{\frac{N}{p}}}
\|\nabla J(a)\|_{\dot{B}_{p,1}^{\frac{N}{p}-1}}\\[1ex]
&\quad\quad\quad\quad\quad\quad\lesssim\|a\|_{\dot{B}_{p,1}^{\frac{N}{p}}}\Big(\|\bu\|_{\dot{B}_{2,1}^{\frac{N}{2}+1}}^\ell
+\|\bu\|_{\dot{B}_{p,1}^{\frac{N}{p}+1}}^h\Big)
\lesssim\mathcal{X}_p(t)\Big(\|\bu\|_{\dot{B}_{2,1}^{\frac{N}{2}+1}}^\ell+\|\bu\|_{\dot{B}_{p,1}^{\frac{N}{p}+1}}^h\Big).
\end{split}
\ee

For the term  $\pi_1(a)(\nabla (I\cdot \bH)-I\cdot \nabla \bH)$, it follows that
\be\label{5.8}
\begin{split}
&\|\pi_1(a)(\nabla (I\cdot \bH)-I\cdot \nabla \bH)\|_{\dot{B}_{p,1}^{\frac{N}{p}-1}}^h\lesssim\|\pi_1(a)\nabla \bH\|_{\dot{B}_{p,1}^{\frac{N}{p}}}^h
\lesssim\|a\|_{\dot{B}_{p,1}^{\frac{N}{p}}}\|\nabla\bH\|_{\dot{B}_{p,1}^{\frac{N}{p}}}\\[1ex]
&\quad\quad\quad\quad\quad\lesssim\|a\|_{\dot{B}_{p,1}^{\frac{N}{p}}}\Big(\|\bH\|_{\dot{B}_{2,1}^{\frac{N}{2}+1}}^\ell
+\|\bH\|_{\dot{B}_{p,1}^{\frac{N}{p}+1}}^h\Big)
\lesssim\mathcal{X}_p(t)\Big(\|\bH\|_{\dot{B}_{2,1}^{\frac{N}{2}+1}}^\ell+\|\bH\|_{\dot{B}_{p,1}^{\frac{N}{p}+1}}^h\Big).
\end{split}
\ee
As for the last term $\frac{1}{1+a}\Big(\frac{1}{2}\nabla |\bH|^2-\bH\cdot \nabla \bH\Big)$ in ${\bf g}$,
we also apply the decomposition \eqref{4.200} to yield that
\be\label{5.9}
\|\bH\cdot((\nabla \bH)^T-\nabla \bH)\|_{\dot{B}_{p,1}^{\frac{N}{p}-1}}^h
\lesssim\|\bH\|_{\dot{B}_{p,1}^{\frac{N}{p}-1}}
\|\nabla\bH\|_{\dot{B}_{p,1}^{\frac{N}{p}}}\lesssim\mathcal{X}_p(t)
\Big(\|\bH\|_{\dot{B}_{2,1}^{\frac{N}{2}+1}}^\ell+\|\bH\|_{\dot{B}_{p,1}^{\frac{N}{p}+1}}^h\Big)
\ee
and
\be\label{5.10}
\begin{split}
\|\pi_1(a)\bH\cdot((\nabla \bH)^T-\nabla \bH)\|_{\dot{B}_{p,1}^{\frac{N}{p}-1}}^h
&\lesssim\|\pi_1(a)\|_{\dot{B}_{p,1}^{\frac{N}{p}}}\|\bH\|_{\dot{B}_{p,1}^{\frac{N}{p}-1}}
\|\nabla\bH\|_{\dot{B}_{p,1}^{\frac{N}{p}}}\\[1ex]
&\lesssim\mathcal{X}_p(t)
\Big(\|\bH\|_{\dot{B}_{2,1}^{\frac{N}{2}+1}}^\ell+\|\bH\|_{\dot{B}_{p,1}^{\frac{N}{p}+1}}^h\Big).
\end{split}
\ee
Finally, for terms in ${\bf m}$, it holds that
\be\label{5.11}
\begin{split}
&\|-\bH({\rm div}\bu)+\bH\cdot\nabla\bu-\bu\cdot\nabla\bH\|_{\dot{B}_{p,1}^{\frac{N}{p}-1}}^h
\lesssim\|(\bu,\bH)\|_{\dot{B}_{p,1}^{\frac{N}{p}-1}}
\|(\nabla\bu,\nabla\bH)\|_{\dot{B}_{p,1}^{\frac{N}{p}}}\\[1ex]
&\quad\quad\quad\quad\quad\quad\quad\quad\quad\quad\quad\quad\quad\quad\quad\lesssim\mathcal{X}_p(t)
\Big(\|(\bu,\bH)\|_{\dot{B}_{2,1}^{\frac{N}{2}+1}}^\ell+\|(\bu,\bH)\|_{\dot{B}_{p,1}^{\frac{N}{p}+1}}^h\Big).
\end{split}
\ee
Combining \eqref{5.4}-\eqref{5.11}, we end up with
\be\label{5.12}
\|({\bf g}, {\bf m})\|_{\dot{B}_{p,1}^{\frac{N}{p}-1}}^h\lesssim\mathcal{X}_p(t)
\Big(\|(a,\bu,\bH)\|_{\dot{B}_{2,1}^{\frac{N}{2}+1}}^\ell
+\|a\|_{\dot{B}_{p,1}^{\frac{N}{p}}}^h+\|(\bu,\bH)\|_{\dot{B}_{p,1}^{\frac{N}{p}+1}}^h\Big).
\ee

In what follows, we bound the low frequency term $\|(f,{\bf g}, {\bf m})\|_{\dot{B}_{2,1}^{\frac{N}{2}-1}}^\ell$ in the right hand of \eqref{5.1}, which
has a little bit more difficult. Let us first introduce the following two inequalities:
\be\label{5.14}
\|T_fg\|_{\dot{B}_{2,1}^{\frac{N}{2}-1}}\lesssim\|f\|_{\dot{B}_{p,1}^{s}}\|{ g}\|_{\dot{B}_{p,1}^{-s+\frac{2N}{p}-1}}
\ee
if $s\leq \frac{2N}{p}-\frac{N}{2}$ and $2\leq p\leq 4$, and
\be\label{5.15}
\|R(f,g)\|_{\dot{B}_{2,1}^{\frac{N}{2}-1}}\lesssim\|f\|_{\dot{B}_{p,1}^{s}}\|{ g}\|_{\dot{B}_{p,1}^{-s+\frac{2N}{p}-1}}
\ee
if $N\geq 2$  and  $2\leq p\leq 4$.
\begin{proof}[Proof of \eqref{5.14}]
Set $\frac{1}{p^\ast}+\frac{1}{p}=1$ and $\|c(j)\|_{l^1}=1$. From the definition of $T_fg$, we obtain
\begin{equation}\label{5.16}
\begin{split}
\|\dot{\Delta}_j(T_fg)\|_{L^2}&\leq \sum_{|k-j|\leq 4}\sum_{k^\prime\leq k-2}\|\dot{\Delta}_j(\dot{\Delta}_{k^\prime}f\dot{\Delta}_kg)\|_{L^2}\leq\sum_{|k-j|\leq 4}\sum_{k^\prime\leq k-2}\|\dot{\Delta}_{k^\prime}f\|_{L^{p^\ast}}\|\dot{\Delta}_kg\|_{L^p}\\[1ex]
&\lesssim\sum_{|k-j|\leq 4}\sum_{k^\prime\leq k-2}2^{k^\prime(\frac{2N}{p}-\frac{N}{2})}\|\dot{\Delta}_{k^\prime}f\|_{L^p}\|\dot{\Delta}_kg\|_{L^p}\\[1ex]
&\lesssim\sum_{|k-j|\leq 4}\sum_{k^\prime\leq k-2}2^{k^\prime(\frac{2N}{p}-\frac{N}{2}-s)}2^{k^\prime s}
\|\dot{\Delta}_{k^\prime}f\|_{L^p}2^{k(s-\frac{2N}{p}+1)}2^{k(-s+\frac{2N}{p}-1)}
\|\dot{\Delta}_kg\|_{L^p}\\[1ex]
&\lesssim c(j)2^{j(1-\frac{N}{2})}
\|f\|_{\dot{B}_{p,1}^{s}}\|g\|_{\dot{B}_{p,1}^{-s+\frac{2N}{p}-1}},
\end{split}
\end{equation}
which yields \eqref{5.14}. Where we used that $p^\ast\geq p$ if $2\leq p\leq 4$ in the third inequality, and the condition $\frac{2N}{p}-\frac{N}{2}-s\geq 0$  in the last inequality.
\end{proof}

\begin{proof}[Proof of \eqref{5.15}] From the definition of $R(f,g)$, it follows that
\begin{equation}\label{5.17}
\begin{split}
&\quad\|\dot{\Delta}_jR(f,g)\|_{L^2}\leq \sum_{k\geq j-3}\sum_{|k-k^\prime|\leq 1}
\|\dot{\Delta}_j(\dot{\Delta}_{k}f\dot{\Delta}_{k^\prime}g)\|_{L^2}\\[1ex]
&\lesssim 2^{j(\frac{2N}{p}-\frac{N}{2})}\sum_{k\geq j-3}\sum_{|k-k^\prime|\leq 1}2^{k(-s)}2^{ks}\|\dot{\Delta}_{k}f\|_{L^{p}}2^{k^\prime(s-\frac{2N}{p}+1)}
2^{k^\prime(-s+\frac{2N}{p}-1)}\|\dot{\Delta}_{k^\prime}g\|_{L^{p}}\\[1ex]
&\lesssim2^{j(\frac{2N}{p}-\frac{N}{2})}\sum_{k\geq j-3}2^{k(-\frac{2N}{p}+1)}c^2(k)\|f\|_{\dot{B}_{p,1}^{s}}
\|g\|_{\dot{B}_{p,1}^{-s+\frac{2N}{p}-1}}\\[1ex]
&\lesssim c(j)2^{j(1-\frac{N}{2})}\|f\|_{\dot{B}_{p,1}^{s}}
\|g\|_{\dot{B}_{p,1}^{-s+\frac{2N}{p}-1}},
\end{split}
\end{equation}
which yields \eqref{5.15}. Where we use that $1\leq \frac{p}{2}\leq 2$  in the second inequality and $1-\frac{2N}{p}\leq 0$ in the last inequality.
\end{proof}

 We claim that
\be\label{5.13}
\|(f,{\bf g}, {\bf m})\|_{\dot{B}_{2,1}^{\frac{N}{2}-1}}^\ell\lesssim\mathcal{X}_p(t)\Big(\|(a,\bu,\bH)\|_{\dot{B}_{2,1}^{\frac{N}{2}+1}}^\ell
+\|a\|_{\dot{B}_{p,1}^{\frac{N}{p}}}^h+\|(\bu,\bH)\|_{\dot{B}_{p,1}^{\frac{N}{p}+1}}^h\Big).
\ee
In what follows, we will prove \eqref{5.13} and  inequalities \eqref{5.14} and \eqref{5.15} are often used for the purpose.

\underline{Estimate of $a{\rm div}\bu$}. Decomposing
$$a{\rm div}\bu=T_a{\rm div}\bu+R(a,{\rm div}\bu)+T_{{\rm div}\bu}a^\ell+T_{{\rm div}\bu}a^h,$$
one has
\be\label{5.18}
\|T_a{\rm div}\bu+R(a,{\rm div}\bu)\|_{\dot{B}_{2,1}^{\frac{N}{2}-1}}^\ell\lesssim\|a\|_{\dot{B}_{p,1}^{\frac{N}{p}-1}}
\|{\rm div}\bu\|_{\dot{B}_{p,1}^{\frac{N}{p}}}\lesssim\|a\|_{\dot{B}_{p,1}^{\frac{N}{p}-1}}
\|\bu\|_{\dot{B}_{p,1}^{\frac{N}{p}+1}}.
\ee
\be\label{5.19}
\|T_{{\rm div}\bu}a^\ell\|_{\dot{B}_{2,1}^{\frac{N}{2}-1}}^\ell\lesssim\|{\rm div}\bu\|_{L^\infty}
\|a^\ell\|_{\dot{B}_{2,1}^{\frac{N}{2}-1}}\lesssim\|\bu\|_{\dot{B}_{p,1}^{\frac{N}{p}+1}}
\|a\|_{\dot{B}_{2,1}^{\frac{N}{2}-1}}^\ell.
\ee
To handle the last term in the decomposition of $a{\rm div}\bu$, we observe that owing to the spectral cut-off, there exists
a universal integer $N_0$ such that
$$
(T_{{\rm div}\bu}a^h)^\ell=\dot{S}_{k_0+1}\Big(\sum_{|k-k_0|\leq N_0}\dot{S}_{k-1}({\rm div}\bu)\dot{\Delta}_ka^h\Big).
$$
Thus, one has
\be\label{5.20}
\begin{split}
&\quad\|T_{{\rm div}\bu}a^h\|_{\dot{B}_{2,1}^{\frac{N}{2}-1}}^\ell\approx 2^{k_0(\frac{N}{2}-1)}\sum_{|k-k_0|\leq N_0}
\|\dot{S}_{k-1}({\rm div}\bu)\dot{\Delta}_ka^h\|_{L^2}\\[1ex]
&\lesssim 2^{k_0(\frac{N}{2}-1)}\sum_{|k-k_0|\leq N_0}\|\dot{S}_{k-1}({\rm div}\bu)\|_{L^{p^\ast}}\|\dot{\Delta}_ka^h\|_{L^p}\\[1ex]
&\lesssim
2^{k_0(\frac{N}{2}-1)}\sum_{|k-k_0|\leq N_0}\sum_{k^\prime\leq k-2}\|\dot{\Delta}_{k^\prime}({\rm div}\bu)\|_{L^{p^\ast}}\|\dot{\Delta}_ka^h\|_{L^p}\\[1ex]
&\lesssim 2^{k_0(\frac{N}{2}-1)}\sum_{|k-k_0|\leq N_0}\sum_{k^\prime\leq k-2}2^{k^\prime(\frac{2N}{p}-\frac{N}{2})}
\|\dot{\Delta}_{k^\prime}({\rm div}\bu)\|_{L^{p}}\|\dot{\Delta}_ka^h\|_{L^p}\\[1ex]
&\lesssim 2^{k_0(\frac{N}{2}-1)}\sum_{|k-k_0|\leq N_0}\sum_{k^\prime\leq k-2}2^{k^\prime(\frac{2N}{p}-\frac{N}{2}+2-\frac{N}{p})}
2^{k^\prime(\frac{N}{p}-2)}
\|\dot{\Delta}_{k^\prime}({\rm div}\bu)\|_{L^{p}}2^{k(-\frac{N}{p})}2^{k\frac{N}{p}}\|\dot{\Delta}_ka^h\|_{L^p}\\[1ex]
&\lesssim 2^{k_0}\|\bu\|_{\dot{B}_{p,1}^{\frac{N}{p}-1}}\|a\|_{\dot{B}_{p,1}^{\frac{N}{p}}}^h,
\end{split}
\ee
where we have used that $2+\frac{N}{p}-\frac{N}{2}>0$ in the last inequality since $p\leq \frac{2N}{N-2}$.

\underline{Estimate of $\bu\cdot\nabla a$}. We also decompose
$$\bu\cdot\nabla a=T_{\bu}\nabla a^\ell+T_{\bu}\nabla a^h+R(\bu,\nabla a)+T_{\nabla a}\bu,$$
and obtain from \eqref{5.14} and \eqref{5.15}  that
\be\label{5.21}
\begin{split}
\|T_{\nabla a}\bu+R(\bu,\nabla a)\|_{\dot{B}_{2,1}^{\frac{N}{2}-1}}^\ell\lesssim\|\nabla a\|_{\dot{B}_{p,1}^{\frac{N}{p}-2}}
\|\bu\|_{\dot{B}_{p,1}^{\frac{N}{p}+1}}\lesssim\Big(\|a\|_{\dot{B}_{2,1}^{\frac{N}{2}-1}}^\ell
+\|a\|_{\dot{B}_{p,1}^{\frac{N}{p}}}^h\Big)\|\bu\|_{\dot{B}_{p,1}^{\frac{N}{p}+1}},
\end{split}
\ee
\be\label{5.22}
\begin{split}
\|T_{\bu}\nabla a^\ell\|_{\dot{B}_{2,1}^{\frac{N}{2}-1}}\lesssim\|\bu\|_{\dot{B}_{p,1}^{\frac{N}{p}-1}}
\|\nabla a^\ell\|_{\dot{B}_{p,1}^{\frac{N}{p}}}\lesssim\|\bu\|_{\dot{B}_{p,1}^{\frac{N}{p}-1}}
\|a\|_{\dot{B}_{2,1}^{\frac{N}{2}+1}}^\ell.
\end{split}
\ee
Similar to \eqref{5.20}, it holds that
\be\label{5.23}
\begin{split}
&\quad\|T_{\bu}\nabla a^h\|_{\dot{B}_{2,1}^{\frac{N}{2}-1}}^\ell\approx 2^{k_0(\frac{N}{2}-1)}\sum_{|k-k_0|\leq N_0}
\|\dot{S}_{k-1}\bu\dot{\Delta}_k\nabla a^h\|_{L^2}\\[1ex]
&\lesssim 2^{k_0(\frac{N}{2}-1)}\sum_{|k-k_0|\leq N_0}\|\dot{S}_{k-1}\bu\|_{L^{p^\ast}}\|\dot{\Delta}_k\nabla a^h\|_{L^p}\\[1ex]
&\lesssim 2^{k_0(\frac{N}{2}-1)}\sum_{|k-k_0|\leq N_0}\sum_{k^\prime\leq k-2}2^{k^\prime(\frac{2N}{p}-\frac{N}{2}+1-\frac{N}{p})}
2^{k^\prime(\frac{N}{p}-1)}
\|\dot{\Delta}_{k^\prime}\bu\|_{L^{p}}2^{k(1-\frac{N}{p})}2^{k(\frac{N}{p}-1)}\|\dot{\Delta}_k\nabla a^h\|_{L^p}\\[1ex]
&\lesssim 2^{k_0}\|\bu\|_{\dot{B}_{p,1}^{\frac{N}{p}-1}}\|a\|_{\dot{B}_{p,1}^{\frac{N}{p}}}^h,
\end{split}
\ee
where we have used that $1+\frac{N}{p}-\frac{N}{2}\geq 0$ in the last inequality since $p\leq \frac{2N}{N-2}$.

\underline{Estimate of $\bu\cdot\nabla \bu$}. Similarly, we decompose $\bu\cdot\nabla \bu=T_{\bu}\nabla \bu+R(\bu,\nabla\bu)+T_{\nabla\bu}\bu$ and get that
\be\label{5.24}
\begin{split}
\|T_{\bu}\nabla\bu+R(\bu,\nabla \bu)\|_{\dot{B}_{2,1}^{\frac{N}{2}-1}}^\ell\lesssim\|\bu\|_{\dot{B}_{p,1}^{\frac{N}{p}-1}}
\|\nabla\bu\|_{\dot{B}_{p,1}^{\frac{N}{p}}}\lesssim\Big(\|\bu\|_{\dot{B}_{2,1}^{\frac{N}{2}-1}}^\ell
+\|\bu\|_{\dot{B}_{p,1}^{\frac{N}{p}-1}}^h\Big)\|\bu\|_{\dot{B}_{p,1}^{\frac{N}{p}+1}},
\end{split}
\ee
\be\label{5.25}
\begin{split}
\|T_{\nabla\bu}\bu\|_{\dot{B}_{2,1}^{\frac{N}{2}-1}}\lesssim\|\nabla\bu\|_{\dot{B}_{p,1}^{\frac{N}{p}-2}}
\|\bu\|_{\dot{B}_{p,1}^{\frac{N}{p}+1}}\lesssim\Big(\|\bu\|_{\dot{B}_{2,1}^{\frac{N}{2}-1}}^\ell
+\|\bu\|_{\dot{B}_{p,1}^{\frac{N}{p}-1}}^h\Big)\|\bu\|_{\dot{B}_{p,1}^{\frac{N}{p}+1}}.
\end{split}
\ee

\underline{Estimate of $\pi_1(a)\mathcal{A}\bu$}. Decomposing
$$
\pi_1(a)\mathcal{A}\bu=T_{\mathcal{A}\bu}\pi_1(a)
+R(\pi_1(a), \mathcal{A}\bu)+T_{\pi_1(a)}\mathcal{A}\bu^\ell+T_{\pi_1(a)}\mathcal{A}\bu^h,
$$
we have
\be\label{5.26}
\begin{split}
\|T_{\mathcal{A}\bu}\pi_1(a)+R(\mathcal{A}\bu,\pi_1(a)\|_{\dot{B}_{2,1}^{\frac{N}{2}-1}}^\ell\lesssim\|\mathcal{A}\bu \|_{\dot{B}_{p,1}^{\frac{N}{p}-1}}
\|\pi_1(a)\|_{\dot{B}_{p,1}^{\frac{N}{p}}}\lesssim\|a\|_{\dot{B}_{p,1}^{\frac{N}{p}}}
\|\bu\|_{\dot{B}_{p,1}^{\frac{N}{p}+1}},
\end{split}
\ee
\be\label{5.27}
\begin{split}
\|T_{\pi_1(a)}\mathcal{A}\bu^\ell\|_{\dot{B}_{2,1}^{\frac{N}{2}-1}}\lesssim
\|\pi_1(a)\|_{L^\infty}\|\mathcal{A}\bu \|_{\dot{B}_{2,1}^{\frac{N}{2}-1}}^\ell\lesssim\|a\|_{\dot{B}_{p,1}^{\frac{N}{p}}}
\|\bu\|_{\dot{B}_{2,1}^{\frac{N}{2}+1}}^\ell,
\end{split}
\ee
Similar to \eqref{5.20} again, it follows that
\be\label{5.28}
\begin{split}
&\quad\|T_{\pi_1(a)}\mathcal{A}\bu^h\|_{\dot{B}_{2,1}^{\frac{N}{2}-1}}^\ell\approx 2^{k_0(\frac{N}{2}-1)}\sum_{|k-k_0|\leq N_0}
\|\dot{S}_{k-1}\pi_1(a)\dot{\Delta}_k\mathcal{A}\bu^h\|_{L^2}\\[1ex]
&\lesssim 2^{k_0(\frac{N}{2}-1)}\sum_{|k-k_0|\leq N_0}\sum_{k^\prime\leq k-2}\|\dot{\Delta}_{k^\prime}\pi_1(a)\|_{L^{p^\ast}}\|\dot{\Delta}_k\mathcal{A}\bu^h\|_{L^p}\\[1ex]
&\lesssim 2^{k_0(\frac{N}{2}-1)}\sum_{|k-k_0|\leq N_0}\sum_{k^\prime\leq k-2}2^{k^\prime(\frac{2N}{p}-\frac{N}{2}+1-\frac{N}{p})}
2^{k^\prime(\frac{N}{p}-1)}
\|\dot{\Delta}_{k^\prime}\pi_1(a)\|_{L^{p}}2^{k(1-\frac{N}{p})}2^{k(\frac{N}{p}-1)}\|\dot{\Delta}_k\mathcal{A}\bu^h\|_{L^p}\\[1ex]
&\lesssim 2^{k_0}\|\pi_1(a)\|_{\dot{B}_{p,1}^{\frac{N}{p}-1}}\|\mathcal{A}\bu\|_{\dot{B}_{p,1}^{\frac{N}{p}-1}}^h
\lesssim \Big(1+\|a\|_{\dot{B}_{p,1}^{\frac{N}{p}}}\Big)\|a\|_{\dot{B}_{p,1}^{\frac{N}{p}-1}}
\|\bu\|_{\dot{B}_{p,1}^{\frac{N}{p}+1}}^h\\[1ex]
&\lesssim\Big(1+\|a\|_{\dot{B}_{2,1}^{\frac{N}{2}-1}}^\ell+\|a\|_{\dot{B}_{p,1}^{\frac{N}{p}}}^h\Big)
\Big(\|a\|_{\dot{B}_{2,1}^{\frac{N}{2}-1}}^\ell+\|a\|_{\dot{B}_{p,1}^{\frac{N}{p}}}^h\Big)
\|\bu\|_{\dot{B}_{p,1}^{\frac{N}{p}+1}}^h.
\end{split}
\ee

\underline{Estimate of $\pi_2(a)\nabla a$}. Decomposing
$$
\pi_2(a)\nabla a=T_{\nabla a}\pi_2(a)
+R(\pi_2(a), \nabla a)+T_{\pi_2(a)}\nabla a^\ell+T_{\pi_2(a)}\nabla a^h,
$$
we obtain
\be\label{5.29}
\begin{split}
&\quad\|T_{\nabla a}\pi_2(a)+R(\nabla a,\pi_2(a)\|_{\dot{B}_{2,1}^{\frac{N}{2}-1}}\lesssim\|\nabla a \|_{\dot{B}_{p,1}^{\frac{N}{p}-1}}
\|\pi_2(a)\|_{\dot{B}_{p,1}^{\frac{N}{p}}}\\[1ex]
&\lesssim\|a\|_{\dot{B}_{p,1}^{\frac{N}{p}}}^2\lesssim
\Big(\|a\|_{\dot{B}_{2,1}^{\frac{N}{2}}}^\ell\Big)^2+\Big(\|a\|_{\dot{B}_{p,1}^{\frac{N}{p}}}^h\Big)^2
\lesssim \|a\|_{\dot{B}_{2,1}^{\frac{N}{2}-1}}^\ell \|a\|_{\dot{B}_{2,1}^{\frac{N}{2}+1}}^\ell+\Big(\|a\|_{\dot{B}_{p,1}^{\frac{N}{p}}}^h\Big)^2,
\end{split}
\ee
\be\label{5.30}
\begin{split}
\|T_{\pi_2(a)}\nabla a^\ell\|_{\dot{B}_{2,1}^{\frac{N}{2}-1}}\lesssim
\|\pi_2(a)\|_{\dot{B}_{p,1}^{\frac{N}{p}-1}}\|\nabla a^\ell\|_{\dot{B}_{p,1}^{\frac{N}{p}}}\lesssim\Big(1+\|a\|_{\dot{B}_{p,1}^{\frac{N}{p}}}\Big)\|a\|_{\dot{B}_{p,1}^{\frac{N}{p}-1}}
\|a\|_{\dot{B}_{2,1}^{\frac{N}{2}+1}}^\ell.
\end{split}
\ee
Similar to \eqref{5.28}, it follows that
\be\label{5.31}
\begin{split}
&\quad\|T_{\pi_2(a)}\nabla a^h\|_{\dot{B}_{2,1}^{\frac{N}{2}-1}}^\ell\approx 2^{k_0(\frac{N}{2}-1)}\sum_{|k-k_0|\leq N_0}
\|\dot{S}_{k-1}\pi_2(a)\dot{\Delta}_k\nabla a^h\|_{L^2}\\[1ex]
&\lesssim 2^{k_0(\frac{N}{2}-1)}\sum_{|k-k_0|\leq N_0}\sum_{k^\prime\leq k-2}2^{k^\prime(\frac{2N}{p}-\frac{N}{2}+1-\frac{N}{p})}
2^{k^\prime(\frac{N}{p}-1)}
\|\dot{\Delta}_{k^\prime}\pi_2(a)\|_{L^{p}}2^{k(1-\frac{N}{p})}2^{k(\frac{N}{p}-1)}\|\dot{\Delta}_k\nabla a^h\|_{L^p}\\[1ex]
&\lesssim 2^{k_0}\|\pi_2(a)\|_{\dot{B}_{p,1}^{\frac{N}{p}-1}}\|\nabla a\|_{\dot{B}_{p,1}^{\frac{N}{p}-1}}^h
\lesssim \Big(1+\|a\|_{\dot{B}_{p,1}^{\frac{N}{p}}}\Big)\|a\|_{\dot{B}_{p,1}^{\frac{N}{p}-1}}
\|a\|_{\dot{B}_{p,1}^{\frac{N}{p}}}^h\\[1ex]
&\lesssim\Big(1+\|a\|_{\dot{B}_{2,1}^{\frac{N}{2}-1}}^\ell+\|a\|_{\dot{B}_{p,1}^{\frac{N}{p}}}^h\Big)
\Big(\|a\|_{\dot{B}_{2,1}^{\frac{N}{2}-1}}^\ell+\|a\|_{\dot{B}_{p,1}^{\frac{N}{p}}}^h\Big)
\|a\|_{\dot{B}_{p,1}^{\frac{N}{p}}}^h.
\end{split}
\ee

\underline{Estimate of $\frac{1}{1+a}(2\widetilde{\mu}(a){\rm div}D(\bu)+\widetilde{\lambda}(a)\nabla{\rm div}\bu)$}.
The estimate of this term could be  performed  similar to that of  $\pi_1(a)\mathcal{A}\bu$ and the details are omitted here.

\underline{Estimate of $\frac{1}{1+a}(2\widetilde{\mu}^\prime(a)D(\bu)\cdot\nabla a+\widetilde{\lambda}^\prime(a){\rm div}\bu\nabla a)$}. We only deal with the term $\frac{2\widetilde{\mu}^\prime(a)}{1+a}D(\bu)\cdot\nabla a$  and the remainder term could be  similarly handled. Recalling \eqref{4.1000}, we derive
\be\label{5.32}
\begin{split}
&\quad\|T_{\nabla J(a)}D\bu+R(\nabla (J(a)),D\bu)\|_{\dot{B}_{2,1}^{\frac{N}{2}-1}}\lesssim\|\nabla J(a) \|_{\dot{B}_{p,1}^{\frac{N}{p}-1}}
\|D\bu\|_{\dot{B}_{p,1}^{\frac{N}{p}}}\lesssim
\|a\|_{\dot{B}_{p,1}^{\frac{N}{p}}}
\|\bu\|_{\dot{B}_{p,1}^{\frac{N}{p}+1}},
\end{split}
\ee
\be\label{5.33}
\begin{split}
&\quad\|T_{D\bu}\nabla J(a)^\ell\|_{\dot{B}_{2,1}^{\frac{N}{2}-1}}\lesssim\|D\bu\|_{\dot{B}_{p,1}^{\frac{N}{p}-2}}\|\nabla J(a)\|_{\dot{B}_{p,1}^{\frac{N}{p}+1}}^\ell
\lesssim
\|\bu\|_{\dot{B}_{p,1}^{\frac{N}{p}-1}}
\|a\|_{\dot{B}_{2,1}^{\frac{N}{2}+1}}^\ell,
\end{split}
\ee
and
\be\label{5.34}
\begin{split}
&\quad\|T_{D\bu}\nabla J(a)^h\|_{\dot{B}_{2,1}^{\frac{N}{2}-1}}^\ell\approx 2^{k_0(\frac{N}{2}-1)}\sum_{|k-k_0|\leq N_0}
\|\dot{S}_{k-1}D\bu\dot{\Delta}_k\nabla J(a)^h\|_{L^2}\\[1ex]
&\lesssim 2^{k_0(\frac{N}{2}-1)}\sum_{|k-k_0|\leq N_0}\sum_{k^\prime\leq k-2}\|\dot{\Delta}_{k^\prime}D\bu\|_{L^{p^\ast}}\|\dot{\Delta}_k\nabla J(a)^h\|_{L^p}\\[1ex]
&\lesssim 2^{k_0(\frac{N}{2}-1)}\sum_{|k-k_0|\leq N_0}\sum_{k^\prime\leq k-2}2^{k^\prime(\frac{2N}{p}-\frac{N}{2}+2-\frac{N}{p})}
2^{k^\prime(\frac{N}{p}-2)}
\|\dot{\Delta}_{k^\prime}D\bu\|_{L^{p}}2^{k(1-\frac{N}{p})}2^{k(\frac{N}{p}-1)}\|\dot{\Delta}_k\nabla J(a)^h\|_{L^p}\\[1ex]
&\lesssim 2^{k_0}\|D\bu\|_{\dot{B}_{p,1}^{\frac{N}{p}-2}}\|\nabla J(a)\|_{\dot{B}_{p,1}^{\frac{N}{p}-1}}^h
\lesssim \|\bu\|_{\dot{B}_{p,1}^{\frac{N}{p}-1}}
\|a\|_{\dot{B}_{p,1}^{\frac{N}{p}}}^h.
\end{split}
\ee

\underline{Estimate of $\pi_1(a)(\nabla (I\cdot \bH)-I\cdot \nabla \bH)$}. Recall that  we decompose
$\pi_1(a)=\pi_1^\prime(0)a+\bar{\pi}_1(a)a$. The estimate of $a(\nabla (I\cdot \bH)-I\cdot \nabla \bH)$
is similar to that of  $a{\rm div}\bu$ and we omit it here. The remaining term can be estimated as follows.
\be\label{5.35}
\begin{split}
\|T_{\bar{\pi}_1(a)a}\nabla \bH+R(\bar{\pi}_1(a)a,\nabla \bH)\|_{\dot{B}_{2,1}^{\frac{N}{2}-1}}
&\lesssim\|\bar{\pi}_1(a)a\|_{\dot{B}_{p,1}^{\frac{N}{p}-1}}
\|\nabla\bH\|_{\dot{B}_{p,1}^{\frac{N}{p}}}\\[1ex]
&\lesssim
\|a\|_{\dot{B}_{p,1}^{\frac{N}{p}}}\|a\|_{\dot{B}_{p,1}^{\frac{N}{p}-1}}
\|\bH\|_{\dot{B}_{p,1}^{\frac{N}{p}+1}},
\end{split}
\ee
\be\label{5.36}
\begin{split}
&\quad\|T_{\nabla\bH}\bar{\pi}_1(a)a\|_{\dot{B}_{2,1}^{\frac{N}{2}-1}}
\lesssim
\|\nabla\bH\|_{L^\infty}\|\bar{\pi}_1(a)a\|_{\dot{B}_{2,1}^{\frac{N}{2}-1}}\\[1ex]
&\lesssim
\|\nabla\bH\|_{\dot{B}_{p,1}^{\frac{N}{p}}}\Big(\|T_{\bar{\pi}_1(a)}a+R(\bar{\pi}_1(a), a)\|_{\dot{B}_{2,1}^{\frac{N}{2}-1}}+\|T_{a}\bar{\pi}_1(a)\|_{\dot{B}_{2,1}^{\frac{N}{2}-1}}\Big)\\[1ex]
&\lesssim\|\bH\|_{\dot{B}_{p,1}^{\frac{N}{p}+1}}\Big(\|\bar{\pi}_1(a)\|_{\dot{B}_{p,1}^{\frac{N}{p}-1}}
\|a\|_{\dot{B}_{p,1}^{\frac{N}{p}}}
+\|a\|_{\dot{B}_{p,1}^{\frac{N}{p}-1}}\|\bar{\pi}_1(a)\|_{\dot{B}_{p,1}^{\frac{N}{p}}}\Big)\\[1ex]
&\lesssim\|\bH\|_{\dot{B}_{p,1}^{\frac{N}{p}+1}}\|a\|_{\dot{B}_{p,1}^{\frac{N}{p}}}
\|a\|_{\dot{B}_{p,1}^{\frac{N}{p}-1}}\Big(1+\|a\|_{\dot{B}_{p,1}^{\frac{N}{p}}}\Big).
\end{split}
\ee

\underline{Estimate of $\frac{1}{1+a}(\frac{1}{2}\nabla |\bH|^2-\bH\cdot \nabla \bH)$}. By \eqref{4.200},  the term with $\bH\cdot((\nabla \bH)^T-\nabla \bH)$ may be handled  similar to  $\bu\cdot\nabla\bu$, and the term with $\pi_1(a)\bH\cdot((\nabla \bH)^T-\nabla \bH)$ would be estimated as follows.
\be\label{5.37}
\begin{split}
\|T_{{\pi}_1(a)\bH}\nabla \bH+R({\pi}_1(a)\bH,\nabla \bH)\|_{\dot{B}_{2,1}^{\frac{N}{2}-1}}
&\lesssim\|{\pi}_1(a)\bH\|_{\dot{B}_{p,1}^{\frac{N}{p}-1}}
\|\nabla\bH\|_{\dot{B}_{p,1}^{\frac{N}{p}}}\\[1ex]
&\lesssim
\|a\|_{\dot{B}_{p,1}^{\frac{N}{p}}}\|\bH\|_{\dot{B}_{p,1}^{\frac{N}{p}-1}}
\|\bH\|_{\dot{B}_{p,1}^{\frac{N}{p}+1}},
\end{split}
\ee
\be\label{5.38}
\begin{split}
\|T_{\nabla\bH}{\pi}_1(a)\bH\|_{\dot{B}_{2,1}^{\frac{N}{2}-1}}
&\lesssim
\|\nabla\bH\|_{\dot{B}_{p,1}^{\frac{N}{p}-1}}\|{\pi}_1(a)\bH\|_{\dot{B}_{p,1}^{\frac{N}{p}}}\\[1ex]
&\lesssim\|\bH\|_{\dot{B}_{p,1}^{\frac{N}{p}}}^2\|{\pi}_1(a)\|_{\dot{B}_{p,1}^{\frac{N}{p}}}
\lesssim \|\bH\|_{\dot{B}_{p,1}^{\frac{N}{p}-1}}\|\bH\|_{\dot{B}_{p,1}^{\frac{N}{p}+1}}\|a\|_{\dot{B}_{p,1}^{\frac{N}{p}}}.
\end{split}
\ee

\underline{Estimate of ${\bf m}$}.
The estimation of ${\bf m}$ is  similar to that of  $\bu\cdot\nabla \bu$ and the details are omitted.
So far, the inequality \eqref{5.13} is proved.

Inserting \eqref{5.2},\eqref{5.3},\eqref{5.12} and \eqref{5.13} into \eqref{5.1} and applying the fact that $\mathcal{X}_p(t)\lesssim\mathcal{X}_{p,0}\ll 1$ for all $t\geq 0$, we end up with
\be\label{5.39}
\begin{split}
  &\quad\frac{d}{dt}\Big(\|(a, \bu, \bH)\|_{\dot{B}_{2,1}^{\frac{N}{2}-1}}^\ell+\|(\nabla a, \bu, \bH)\|_{\dot{B}_{p,1}^{\frac{N}{p}-1}}^h\Big)\\[1ex]
  &\quad\quad\quad\quad\quad\quad\quad\quad\quad\quad
  +\Big(\|(a, \bu, \bH)\|_{\dot{B}_{2,1}^{\frac{N}{2}+1}}^\ell
  +\|a\|_{\dot{B}_{p,1}^{\frac{N}{p}}}^h+\|(\bu, \bH)\|_{\dot{B}_{p,1}^{\frac{N}{p}+1}}^h\Big)\leq 0.
\end{split}
\ee

In what follows, we will employ  the following interpolation inequalities:
\begin{prop}\label{pra.7}(\cite{xin2018optimal})
Suppose that $m\neq \rho$. Then it holds that
$$
\|f\|_{\dot{B}_{p,1}^j}^\ell\lesssim (\|f\|_{\dot{B}_{r,\infty}^m}^\ell)^{1-\theta}(\|f\|_{\dot{B}_{r,\infty}^{\rho}}^\ell)^{\theta},\,\,\,\,\,\,
\|f\|_{\dot{B}_{p,1}^j}^h\lesssim (\|f\|_{\dot{B}_{r,\infty}^m}^h)^{1-\theta}(\|f\|_{\dot{B}_{r,\infty}^{\rho}}^h)^{\theta}
$$
where $j+N(\frac{1}{r}-\frac{1}{p})=m(1-\theta)+\rho\theta$ for $0<\theta<1$ and $1\leq r\leq p\leq \infty$.
\end{prop}

Due to $-\sigma_1<\frac{N}{2}-1\leq \frac{N}{p}<\frac{N}{2}+1$, it follows from Proposition \ref{pra.7} that
\be\label{5.40}
\|(a,\bu,\bH)\|_{\dot{B}_{2,1}^{\frac{N}{2}-1}}^\ell
\leq C\Big(\|(a,\bu,\bH)\|_{\dot{B}_{2,\infty}^{-\sigma_1}}^\ell\Big)^{\theta_0}
\Big(\|(a,\bu,\bH)\|_{\dot{B}_{2,\infty}^{\frac{N}{2}+1}}^\ell\Big)^{1-\theta_0},
\ee
where $\theta_0=\frac{2}{N/2+1+\sigma_1}\in (0,1)$. In view of \eqref{4.46}, we have
$$
\|(a,\bu,\bH)\|_{\dot{B}_{2,\infty}^{\frac{N}{2}+1}}^\ell\geq c_0\Big(\|(a,\bu,\bH)\|_{\dot{B}_{2,1}^{\frac{N}{2}-1}}^\ell\Big)^{\frac{1}{1-\theta_0}}
$$
where $c_0=C^{-\frac{1}{1-\theta_0}}C_0^{-\frac{\theta_0}{1-\theta_0}}$. Moreover, it follows from
$\|(\nabla a, \bu, \bH)\|_{\dot{B}_{p,1}^{\frac{N}{p}-1}}^h\leq \mathcal{X}_p(t)\lesssim \mathcal{X}_{p,0}\ll 1$
for all $t\geq 0$ that
$$
\|a\|_{\dot{B}_{p,1}^{\frac{N}{p}}}^h\geq \Big(\|a\|_{\dot{B}_{p,1}^{\frac{N}{p}}}^h\Big)^{\frac{1}{1-\theta_0}},\,\,\,\,\,
\|(\bu, \bH)\|_{\dot{B}_{p,1}^{\frac{N}{p}+1}}^h\geq \Big(\|(\bu, \bH)\|_{\dot{B}_{p,1}^{\frac{N}{p}-1}}^h\Big)^{\frac{1}{1-\theta_0}}.
$$
Thus, there exists a constant $\tilde{c}_0>0$ such that the following Lyapunov-type inequality holds:
\be\label{5.41}
\begin{split}
  &\quad\frac{d}{dt}\Big(\|(a, \bu, \bH)\|_{\dot{B}_{2,1}^{\frac{N}{2}-1}}^\ell+\|(\nabla a, \bu, \bH)\|_{\dot{B}_{p,1}^{\frac{N}{p}-1}}^h\Big)\\[1ex]
  &\quad\quad\quad\quad\quad\quad\quad
  +\tilde{c}_0\Big(\|(a, \bu, \bH)\|_{\dot{B}_{2,1}^{\frac{N}{2}-1}}^\ell
  +\|(\nabla a, \bu, \bH)\|_{\dot{B}_{p,1}^{\frac{N}{p}-1}}^h\Big)^{1+\frac{2}{N/2-1+\sigma_1}}\leq 0.
\end{split}
\ee
Solving \eqref{5.41}  yields
\be\label{5.42}
\begin{split}
&\quad\|(a, \bu, \bH)(t)\|_{\dot{B}_{2,1}^{\frac{N}{2}-1}}^\ell+\|(\nabla a, \bu, \bH)(t)\|_{\dot{B}_{p,1}^{\frac{N}{p}-1}}^h\\[1ex]
&\leq\Big(\mathcal{X}_{p,0}^{-\frac{2}{N/2-1+\sigma_1}}
+\frac{2\tilde{c}_0t}{N/2-1+\sigma_1}\Big)^{-\frac{N/2-1+\sigma_1}{2}}
\lesssim (1+t)^{-\frac{N/2-1+\sigma_1}{2}}
\end{split}
\ee
for all $t\geq 0$. Through the embedding properties in Proposition \ref{pr2.1}, we arrive at
\be\label{5.43}
\begin{split}
\|(a, \bu, \bH)(t)\|_{\dot{B}_{p,1}^{\frac{N}{p}-1}}\lesssim\|(a, \bu, \bH)(t)\|_{\dot{B}_{2,1}^{\frac{N}{2}-1}}^\ell+\|(\nabla a, \bu, \bH)(t)\|_{\dot{B}_{p,1}^{\frac{N}{p}-1}}^h
\lesssim (1+t)^{-\frac{N/2-1+\sigma_1}{2}}.
\end{split}
\ee
In addition, if $\sigma\in (-\sigma_1-N(\frac{1}{2}-\frac{1}{p}), \frac{N}{p}-1)$, then employing Proposition
\ref{pra.7} once again implies that
\be\label{5.44}
\begin{split}
\|(a, \bu, \bH)(t)\|_{\dot{B}_{p,1}^{\sigma}}^\ell\lesssim\|(a, \bu, \bH)(t)\|_{\dot{B}_{2,1}^{\sigma+N(\frac{1}{2}-\frac{1}{p})}}^\ell
\lesssim\Big(\|(a,\bu,\bH)\|_{\dot{B}_{2,\infty}^{-\sigma_1}}^\ell\Big)^{\theta_1}
\Big(\|(a,\bu,\bH)\|_{\dot{B}_{2,\infty}^{\frac{N}{2}-1}}^\ell\Big)^{1-\theta_1},
\end{split}
\ee
where
$$
\theta_1=\frac{\frac{N}{p}-1-\sigma}{\frac{N}{2}-1+\sigma_1}\in(0,1).
$$
Note that
$$
\|(a,\bu,\bH)\|_{\dot{B}_{2,\infty}^{-\sigma_1}}^\ell\leq C_0
$$
for all $t\geq 0$.  From \eqref{5.42} and \eqref{5.44}, we deduce that
\be\label{5.45}
\begin{split}
\|(a, \bu, \bH)(t)\|_{\dot{B}_{p,1}^{\sigma}}^\ell\lesssim
\Big[(1+t)^{-\frac{N/2-1+\sigma_1}{2}}\Big]^{1-\theta_1}
=(1+t)^{-\frac{N}{2}(\frac{1}{2}-\frac{1}{p})-\frac{\sigma+\sigma_1}{2}}
\end{split}
\ee
for all $t\geq 0$, which leads to
\be\label{5.46}
\begin{split}
\|(a, \bu, \bH)(t)\|_{\dot{B}_{p,1}^{\sigma}}\lesssim
\|(a, \bu, \bH)(t)\|_{\dot{B}_{p,1}^{\sigma}}^\ell+\|(a, \bu, \bH)(t)\|_{\dot{B}_{p,1}^{\sigma}}^h
\lesssim (1+t)^{-\frac{N}{2}(\frac{1}{2}-\frac{1}{p})-\frac{\sigma+\sigma_1}{2}}
\end{split}
\ee
for $\sigma\in (-\sigma_1-N(\frac{1}{2}-\frac{1}{p}), \frac{N}{p}-1)$. So far, the proof of Theorem \ref{th1.2} is completed.
\subsection{Proof of Corollary \ref{col1}} In fact, Corollary \ref{col1} can be regarded as the direct consequence of the following interpolation inequality:
\begin{prop}\label{pra.8}(\cite{bahouri2011fourier})
The following interpolation inequality holds true:
$$
\|\Lambda^lf\|_{L^r}\lesssim \|\Lambda^mf\|_{L^q}^{1-\theta}\|\Lambda^kf\|_{L^q}^\theta,
$$
whenever $0\leq \theta\leq 1, 1\leq q\leq r\leq \infty$ and
$$
l+N\Big(\frac{1}{q}-\frac{1}{r}\Big)=m(1-\theta)+k\theta.
$$
\end{prop}

 With the aid of  Proposition \ref{pra.8}, we define $\theta_2$ by the relation
$$
m(1-\theta_2)+k\theta_2=l+N\Big(\frac{1}{p}-\frac{1}{r}\Big),
$$
where $m=\frac{N}{p}-1$ and $k=-\sigma_1-N(\frac{1}{2}-\frac{1}{p})+\varepsilon$ with $\varepsilon>0$ small enough. It is easy to see that $\theta_2\in(0,1)$ if $\varepsilon>0$ is small enough. As a consequence, we conclude  by $\dot{B}_{p,1}^0\hookrightarrow L^p$ that
\be\label{5.47}
\begin{split}
&\quad\|\Lambda^l(a,\bu,\bH)\|_{L^r}\lesssim
\|\Lambda^m(a,\bu,\bH)\|_{L^p}^{1-\theta_2}\|\Lambda^k(a,\bu,\bH)\|_{L^p}^{\theta_2}\\[1ex]
&\lesssim \left[(1+t)^{-\frac{N}{2}(\frac{1}{2}-\frac{1}{p})-\frac{m+\sigma_1}{2}}\right]^{1-\theta_2}
\left[(1+t)^{-\frac{N}{2}(\frac{1}{2}-\frac{1}{p})-\frac{k+\sigma_1}{2}}\right]^{\theta_2}
=(1+t)^{-\frac{N}{2}(\frac{1}{2}-\frac{1}{r})-\frac{l+\sigma_1}{2}}
\end{split}
\ee
for $p\leq r\leq \infty$ and $l\in\mathbb{R}$ satisfying $-\sigma_1
-\frac{N}{2}+\frac{N}{p}<l+\frac{N}{p}-\frac{N}{r}\leq \frac{N}{p}-1$.
Thus, we finish the proof of Corollary \ref{col1}.

\providecommand{\href}[2]{#2}
\providecommand{\arxiv}[1]{\href{http://arxiv.org/abs/#1}{arXiv:#1}}
\providecommand{\url}[1]{\texttt{#1}}
\providecommand{\urlprefix}{URL }


\end{document}